\documentclass[11pt,a4paper]{article}

\pdfoutput=1 

\usepackage{amsfonts,amsmath,amsthm,amssymb,amstext}
\usepackage{geometry}                
\usepackage{fullpage}
\usepackage{graphicx}
\usepackage{subfigure}
\usepackage{latexsym}

\usepackage[pdftex,bookmarks,colorlinks,breaklinks]{hyperref}  % PDF hyperlinks, with coloured links
\usepackage{color}

\definecolor{dullmagenta}{rgb}{0.4,0,0.4}   % #660066
\definecolor{darkblue}{rgb}{0,0,0.4}
\hypersetup{linkcolor=red,citecolor=blue,filecolor=dullmagenta,urlcolor=darkblue} % coloured links

\newtheorem{theorem}{Theorem}[section]
\newtheorem{lemma}[theorem]{Lemma}

\newtheorem{prop}[theorem]{Proposition}

\newenvironment{remark}%
  {\par\medbreak\refstepcounter{theorem}%
    \noindent\textbf{Remark~\thetheorem. }}%
  {\par\medskip}

\usepackage{algorithm}
\usepackage{algorithmic}

\title{Minimal Dirichlet energy partitions for graphs}

\date{\today} 

\author{Braxton Osting\thanks{Department of Mathematics, University of California, Los
Angeles, CA 90095 USA ({\tt braxton@math.ucla.edu}).}, 
Chris D. White\thanks{Department of Mathematics, University of Texas at Austin, Austin, TX 78712 USA ({\tt cwhite@math.utexas.edu})}, and
\'Edouard Oudet\thanks{LJK, Universit\'e Joseph Fourier, Grenoble, France ({\tt edouard.oudet@imag.fr}). } }

\begin{document}
\maketitle
\begin{abstract} 
Motivated by a geometric problem, we introduce a new non-convex graph partitioning  objective where the optimality criterion is given by the sum of the Dirichlet eigenvalues of the partition components.  A relaxed formulation is identified and a novel rearrangement algorithm is proposed, which we show is strictly decreasing and converges in a finite number of iterations to a local minimum of the relaxed objective function.  Our method is applied to several clustering problems on graphs constructed from synthetic data, MNIST handwritten digits, and manifold discretizations. The model has a semi-supervised extension and provides a natural representative for the clusters as well.  
\end{abstract}

 
\section{Introduction} \label{sec:intro}
Given a graph $G = (V,E)$ with non-negative edge weights $\{w_{e}\}_{e\in E}$, we consider the problem of ``optimally" partitioning  the vertex set, $V$, into $k$ subsets.  This \emph{graph partitioning problem} frequently arises in the machine learning community, where the vertices  represent observed data points and the goal is to identify meaningful groups ({\it i.e.} ``clusters'') within the data.  One difficulty arises in choosing a measure of optimality which is computable and satisfies certain desirable properties which can be application dependent.  

In this paper, we introduce a new measure of optimality for graph partitions, 
based on the sum of the Dirichlet eigenvalues of the partition components,
\begin{equation*} 
\min_{V= \amalg_{i=1}^k V_i}  \ \sum_{i=1}^k  \lambda (V_i).
\end{equation*}
Here $\lambda(V_i)$ denotes the first Dirichlet eigenvalue of the partition component $V_i$; a precise definition is given in \S \ref{sec:model}.  One benefit of the proposed objective is that it is not based on minimizing perimeter, and can thus take the interior of the partition components into account. Also, although we do \emph{not} fix the sizes of the partition components a priori, we will demonstrate that this objective prefers balanced partition component sizes. We assume that the number of partitions, $k$, is known and fixed. 

The eigenvalue partitioning problem as stated is combinatorial. To solve the problem, in \S\ref{sec:rel} we introduce a relaxation and an algorithm to solve the relaxed problem. Our algorithm was motivated by a geometric interpretation of the problem and is novel in that it does \emph{not} rely on a convex approximation or gradient descent method. We show that the algorithm is strictly decreasing and converges in a finite number of iterations to a local minimum of the relaxed objective function. In \S\ref{sec:num}, we  demonstrate with numerical examples some properties of the proposed model.  For example, it can handle arbitrary cluster shapes, which is important in most applications.  Moreover, our algorithm naturally provides confidences for label assignments. Consequently, the vertex with the highest confidence can be interpreted as a  representative for each cluster; the problem of representative-based clustering has many variants \cite{elhamifar2012,frey2007}, and is especially informative in mining image and video data.  In \S\ref{semi} we also describe a semi-supervised extension of our algorithm and apply this to the MNIST handwritten digit dataset in \S \ref{sec:num}.  Another property of our algorithm is that it can produce geometrically meaningful partitions on discretizations of manifolds; we demonstrate this ability in \S\ref{sec:num}, where we apply it to the torus and the sphere.  In both cases, there are open questions concerning the nature of these manifold partitions. 

Our model was motivated by an analogous geometric problem phrased for domains $\Omega\subset\mathbb{R}^{n}$. Indeed, a recent trend in the applied math community is to utilize models and tools derived from PDEs and geometry to motivate  analogous data-driven formulations and methods.  This results in models which are highly interpretable and have well understood properties.  Perhaps the most relevant examples of this are the use of 
Cheeger Cuts for graph partitioning  \cite{SpielmanTeng96,Shi2000,Luxburg:2007} and the use of PDE models in image processing \cite{chan2005}.  Diffusion maps, which are used to embed a dataset into a relatively low-dimensional Euclidean space 
via the eigenpairs of a diffusion operator, 
 are another PDE-based approach and have become a powerful tool for dimensionality reduction and multi-scale analysis of data sets  \cite{coifman2006,nadler2006}.  
 Functionals and flows coming from materials science (\textit{e.g.}, motion by mean curvature) have been successfully applied to image segmentation and inpainting \cite{MBO1992,MBO1993,mbo2012,garcia2013,esedoglu2013} and have inspired research into curvature flows on graphs \cite{Gennip2013}.  We believe our current model fits into this story; the analogous geometric model which inspired this work will be discussed properly in \S\ref{sec:model}. 

\paragraph{Outline.} 
In \S \ref{sec:model}, we introduce the eigenvalue partitioning problem for a graph. 
In \S \ref{sec:rel}, we introduce a relaxed formulation  of the eigenvalue partitioning problem and propose a rearrangement algorithm for its solution. 
In \S \ref{sec:num}, we apply our proposed partitioning algorithm to a number of examples, both data driven and purely geometric. 
We conclude in \S \ref{sec:disc} with a discussion.

\section{The graph partitioning model} \label{sec:model} 
 In this section, we introduce a new graph partitioning model based on the Dirichlet eigenvalues of the partition components. Before we consider the energy for a vertex partition, we first introduce the energy of a vertex subset.  
Define the \emph{Dirichlet energy} of a subset $S \subset V$, 
\begin{equation}
\label{eq:lam}
\lambda(S) := \inf_{\substack{ \| \psi \|_V = 1 \\  \psi |_{S^c}=0}} \ \|\nabla\psi\|^{2}_{w,E} ,
\end{equation}
where $S^c := V \setminus S$  denotes  the complement of $S$, 
%$(\nabla \psi)_{(i,j)} = \psi_i - \psi_j$, 
\begin{align*}
 \|\nabla\psi\|^{2}_{w,E} := \sum_{ (i,j)\in E}w_{ij}(\psi_{i} - \psi_{j})^{2},  \quad
 \|\psi\|^{2}_{S} := \sum_{i\in S}d^r_{i}\psi_{i}^{2}, \quad \text{and} \quad  d_i := \sum_j w_{ij}. 
 \end{align*}
Here, we take the parameter $r\in [0,1]$.  Observe that for a connected component, $S$, of a disconnected graph, $\lambda(S) = 0$.  Roughly speaking, we see that $\|\nabla\psi\|^{2}_{w,E}$ measures how much $\psi$ varies over $S$, with changes across similar vertices (i.e., large $w_{ij}$) being penalized more than changes across dissimilar vertices.  Observe that if $\psi = \chi_S$, then $\| \nabla \psi \|^2_{w,E}$ simply reduces to $\sum_{i \in S, j \in S^c} w_{i,j} \equiv |\partial S|$, implying that 
 \eqref{eq:lam} measures variations across the boundary of $S$ as well.  
Consequently we see that $\lambda(S)$ is a measure of the connectedness of $S$ that takes into account both interior similarity as well as similarity to the rest of the graph.

The value $\lambda = \lambda(S)$ in \eqref{eq:lam} satisfies the following Dirichlet eigenvalue problem in $S$, for some corresponding eigenvector, $\psi = \psi(S)$,
\begin{subequations}
\label{eq:explcitEigEq}
\begin{align}
       &\Delta_r \psi  = \lambda\psi  \qquad  \text{on } S \subset V\\
      &\psi = 0 \qquad \quad \ \ \  \text{on } S^c.
\end{align}
\end{subequations}
The \emph{graph Laplaican}\footnote{We comment that the sign convention for the graph Laplacian is opposite to that used for the continuum Laplacian in most of the PDE literature (in particular, $\Delta_r$ is a positive semidefinite operator). }, 
$\Delta_r$, 
can be written $\Delta_{r} := D^{-r} (D - W) $ where $D$ is the degree matrix and $W$ is the  weight matrix for the graph. 
When $r=0$, this is the  \emph{combinatorial  graph Laplacian} or \emph{unnormalized symmetric graph Laplacian} and when $r=1$ this is the \emph{asymmetric normalized graph Laplacian} or \emph{random walk  graph Laplacian}. 
Thus, $(\lambda,\psi)$ in \eqref{eq:explcitEigEq}  is simply the first  eigenpair of a submatrix of $\Delta_r$. 
In this work, we will take $r\in[0,1]$ to be a fixed parameter, the choice dependent on the application. As such, we suppress the dependency of $\Delta_r$ on $r$. A more complete discussion of these graph objects can be found in  \cite{Chung:1997,Luxburg:2007}.

We define the \emph{energy of a $k$-partition}, $V = \amalg_{i=1}^{k} V_{i}$ to be  
$ \Lambda_k(\amalg_{i=1}^{k} V_{i}) = \sum_{i=1}^k  \lambda (V_i)$ where $\lambda$ is defined in \eqref{eq:lam}. The graph partitioning problem is then formulated as the following optimization problem 
\begin{equation}
\label{eq:DirObjFun} 
\Lambda^*_k := \min_{V= \amalg_{i=1}^k V_i}  \ \sum_{i=1}^k  \lambda (V_i).
\end{equation}
A minimizer always exists, as there are only finitely many partitions of the graph. 

The proposed model \eqref{eq:DirObjFun} is closely connected to several other problems and methods, which we outline now.

\paragraph{An analogous geometric problem.}
The eigenvalue partitioning problem has an analogous geometric formulation. 
Namely, given a bounded open set $\Omega \subset \mathbb R^2$, or more generally a compact manifold, find the partition $\Omega = \amalg_{i=1}^{k} \Omega_{i}$ which attains 
\begin{equation}
\label{eq:ContPart}
\inf_{\Omega = \amalg_{i=1}^{k} \Omega_{i}} \ \ \sum_{i=1}^k \lambda(\Omega_i),
\end{equation}
where $\lambda(\Omega_{i})$ denotes the first Dirichlet-Laplace eigenavalue of $\Omega_{i}$. 
Existence of optimal partitions for \eqref{eq:ContPart} in the class of quasi-open sets was proved in 
 \cite{Bucur1998}. Subsequently, several papers have investigated \eqref{eq:ContPart} and similar problems, focusing on the regularity of partitions, properties of optimal partitions, the asymptotic behavior  of optimal partitions as $k\to \infty$, and computational methods 
\cite{Cafferelli2007,Bonnaillie2007,Bourdin2010,Helffer2009,Helffer2010b,Helffer2010,Oudet2011,Bucur2013}. 
In particular, the relaxation of  the eigenvalue partitioning problem proposed here is analogous to a relaxation of \eqref{eq:ContPart} proposed in \cite{Bourdin2010}. 
The loss of infinitesimal scale on a graph has many consequences for the Dirichlet spectrum and partitioning problem. For example, the following statement is true in the continuum but fails on a graph: Any  eigenvalue of the Laplace-Dirichlet operator is also the first eigenvalue for each of the nodal domains of the eigenfunction.

\paragraph{Cheeger cut partitioning.} 
The Cheeger cut  \cite{Chung:1997} (or balanced cut) for a graph, $G$, is defined
\begin{equation}
\label{eq:Cheeger}
B^*(G) := \min_{S\subset V} \   B(S) 
\quad \text{where} \quad
B(S) := \frac{ |\partial S | }{ \min \left(\text{vol}(S), \text{vol}(S^c)\right)},
\end{equation} 
where $\partial S := \{(i,j)\in E \colon  i\in S \text{ and } j \in S^c \} $ is the edge boundary of $S$, $|\cdot |$ denotes cardinality,  and $\text{vol}(S) = \sum_{i\in S} d_i$. 
The vertex subset $S^*\subset V$ attaining the minimum in    \eqref{eq:Cheeger} is a 2-class vertex partition that is often used for the bipartitioning problem 
  \cite{SpielmanTeng96,Shi2000,Luxburg:2007}. This approach can be used recursively for the $k$-class problem 
  \cite{Bresson2012}. Alternatively, \cite{Bresson2013} mathematically formulates the general $k$-class vertex partitioning problem by generalizing the Cheeger cut \eqref{eq:Cheeger} and solving
$$
B^*_k (G):= \min_{V= \amalg_{\ell=1}^k V_\ell} \ \sum_{\ell=1}^k B(V_\ell). 
$$ 
By considering the test function $\psi = \chi_S$ for any $S\subset V$,  it follows that 
$\lambda(S) \leq \frac{|\partial S|}{\text{vol}(S)} $. Further relationships between the Dirichlet eigenvalues and ``local'' Cheeger cuts are studied in \cite{chung2007}. The Cheeger partitioning problem shares the attribute of having a geometric analogue; the goal consists of partitioning a manifold into $k$ submanifolds, and methods for solving  the Cheeger partitioning problem can be interpreted in terms of a mean curvature flow on a graph \cite{Gennip2013}.

\paragraph{Nonnegative matrix factorization (NMF).} 
Nonnegative matrix factorization is the general algebraic problem of finding a factorization of a matrix $A = \prod_{i}^{K}N_{i}$
where some, or all of the $N_{i}$ are constrained to be nonnegative.  This type of problem naturally arises in variable selection \cite{lee1999,hofmann1999} and clustering \cite{Yang2012}.  A popular approach for clustering applications is to solve 
 $$ 
 \min_{V\in \mathcal X} \ \| W - VV^{T} \|_{F}^{2}, 
\qquad  \text{where } \ \mathcal{X}: = \{V\in\mathbb R^{n\times k}\colon V^{T}V = \text{Id}, V_{ij} \geq 0\}.
 $$
Here $W$ is a similarity matrix constructed from the data, and $k$ is the desired number of clusters, and $ \text{Id}$ is the $k\times k$ identity matrix. The following proposition shows that in certain instances our proposed objective is equivalent to an NMF objective, where the matrix to be factorized is the transpose of the asymmetric graph Laplacian.

\begin{prop} \label{prop:NMF}
Let $\Psi^*: = \left[\psi_{1}^{D} |  \cdots | \psi_{k}^{D}\right]$ be the matrix where the columns are Dirichlet eigenvectors corresponding to the optimal partition for $r=1$.  Then 
$$ D^{1/2} \Psi^* = \arg\min_{U\in\mathcal{M}}\| D^{-1/2} W D^{-1/2}  - UU^{T}\|_{F}^{2},
\quad \text{where} \quad
\mathcal{M} :=\{U\in\mathbb R^{n\times k}\colon U^{T}U = \mathrm{Id}, U_{ij} \geq 0\},
$$ 
$D = \mathrm{diag}(d)$ is the degree matrix and $W$ is the similarity weight matrix.
\end{prop}
\begin{proof}
Let $V$ be a collection of Dirichlet eigenvectors corresponding to some partition.  Then, by definition, we have $ \Delta V = V \text{diag}(\vec{\lambda}) $
where $\text{diag}(\vec{\lambda})$ is a $k\times k$ diagonal matrix, with the Dirichlet eigenvalues along the diagonal.  Moreover, $V$ satisfies $V_{ij}\geq 0$ and $ V^{T}DV=\text{Id}$.  Thus the partitioning problem \eqref{eq:DirObjFun} is equivalent to
\begin{subequations}
\label{eq:objTr}
\begin{align}
\Lambda^{*}_{k} =  \min_{V\in\mathbb{R}^{n\times k}} & \ \text{tr}(V^{T}D\Delta V) \\
 \text{s.t.} & \ V_{ij}\geq 0, \ \ V^{T}DV=\text{Id}.  
\end{align}
\end{subequations}
Using the definition of the graph Laplacian, the objective function can be expanded to 
$$
\text{tr}(V^{T}DV) - \text{tr}(V^{T}W V) =  
\text{tr}(V^{T}DV) - \text{tr}(V^{T}D^{1/2} D^{- 1/2} WD^{-1/2} D^{1/2} V). 
$$
Thus, \eqref{eq:objTr} is equivalent to
\begin{align*}
  \min_{V\in\mathbb{R}^{n\times k}} & \ \| D^{- 1/2} WD^{-1/2} - D^{1/2}VV^{T}D^{1/2}\|_{F}^{2} \\
 \text{s.t.} & \ V_{ij}\geq 0, \ \ V^{T}DV=\text{Id}.  
\end{align*}
After the change of variables $U:=D^{1/2}V$, we arrive at the stated proposition.
\end{proof}

\begin{remark} A variant of Proposition \ref{prop:NMF} can be shown for the $r=0$ Laplacian whenever the graph is regular, as is the case for an unweighted $k$-nearest neighbor graph. 
\end{remark}

We are not the first to connect NMF with spectral-based methods; \cite{ding2005} describes a connection between various spectral clustering objectives and NMF.
The algorithm proposed in \S \ref{sec:rel} for solving \eqref{eq:DirObjFun} is new for this  NMF objective; typical approaches to NMF problems are algebraic and involve finding good convex approximations. 
Our relaxation is not convex, and our algorithm is novel in that it does \emph{not} rely on a convex approximation or gradient descent method.  An interesting future direction of research might be to extend and analyze our geometric algorithm for other NMF objectives.

\section{Relaxation and a rearrangement algorithm} \label{sec:rel}
In this section, we find a relaxation of the graph partitioning problem \eqref{eq:DirObjFun} and introduce an efficient algorithm for solving the relaxed problem.  

For a vertex function, $\phi \colon V \to [0,1]$ and $\alpha>0$, consider the relaxed  energy
\begin{equation} \label{eq:relDL}
\lambda^{\alpha}(\phi) := \inf_{\| \psi \| = 1 } \|\nabla\psi\|^{2}_{w,E} + \alpha \|\psi\|^{2}_{ (1-\phi)}, 
\quad \text{where} \quad
 \|\psi\|^{2}_{f} := \sum_{i\in V}d_{i}^r f_i \psi_{i}^{2}.
\end{equation}
Observe that $\lambda^{\alpha}(\phi)$ in \eqref{eq:relDL}  is the first eigenvalue of the Schr\"odinger operator $\Delta + \alpha(1-\phi)$, and the eigenfunction, $\psi^\alpha$, satisfies 
\begin{equation}
\label{eq:RelaxedExplicitEig}
\left[  \Delta + \alpha(1-\phi) \right] \psi  = \lambda \psi  \qquad \text{ in } V.
 \end{equation}
The minimizer $\psi^\alpha$ is unique up to a scaling and can be chosen to be strictly positive, {\it i.e.}, for all $i\in V$, $\psi_i^\alpha >0$.\footnote{These facts can be obtained by applying the Perron-Frobenius theorem to the matrix $\beta \mathrm{Id} - [\Delta + \alpha(1-\phi)]$ for sufficiently large $\beta$.} Throughout, we will take $\psi^\alpha$ to be positive with $\| \psi^\alpha \|_V=1$.

If $\phi = \chi_S$ is the indicator function for the set $S\subset V$, then we intuitively think of $\lambda^\alpha (\chi_S)$ as  an approximation to $\lambda(S)$.  The following lemma shows that this approximation is exact in the limit that $\alpha\rightarrow \infty$; moreover, as $\alpha$ becomes large the eigenfunction corresponding to $\lambda^\alpha(\chi_S)$ becomes strongly localized on $S$. 
In the continuous case, on may interpret this relaxation as a   ``fictitious domain method'' \cite{Oudet2004,Bourdin2010}. 

 \begin{lemma} \label{lem:Conv}
 For $S\subset V$, 
 $\displaystyle \lim_{\alpha\to\infty} \lambda^{\alpha}(\chi_{S}) = \lambda(S)$ and 
  $\displaystyle \lim_{\alpha\to\infty} \psi^{\alpha}(\chi_{S}) = \psi^D(S)$, where $\psi^D(S)$ is the Dirichlet eigenvector achieving the infimum in \eqref{eq:lam}.
 \end{lemma}
\begin{proof}
A simple computation shows that
$\frac{d\lambda}{d\alpha} = \|\psi\|^{2}_{S^{c}} > 0$, 
where $\psi$ is the corresponding normalized eigenvector.  
Moreover, it is clear that $\lambda^{\alpha}(\chi_{S}) \leq \lambda(S)$. Consequently $\displaystyle \lim_{\alpha\rightarrow\infty} \lambda^{\alpha}(\chi_{S})$ exists and satisfies $\displaystyle \lim_{\alpha\rightarrow\infty} \lambda^{\alpha}(\chi_{S}) \leq \lambda(S)$. 

For the reverse inequality, observe that if we normalize all eigenvectors, then after possibly passing to a subsequence, there exists a $\tilde \psi$ such that  $\psi^{\alpha} \rightarrow \tilde{\psi}$, and $\|\tilde{\psi}\|^{2}_{S^{c}}=0$.  Thus $\tilde{\psi}$ is admissible for the Dirichlet eigenvalue problem in $S$, giving us that $\lambda(S)\leq\displaystyle\lim_{\alpha\rightarrow\infty} \lambda^{\alpha}(\chi_{S})$.

Since the minimizer of the Dirichlet problem is unique, $\tilde \psi = \psi^D$. Thus, the previous argument shows that the only limit point of $\{ \psi^\alpha\}_\alpha$ is $ \psi^D$ and so $\displaystyle \lim_{\alpha\to\infty} \psi^{\alpha} = \psi^D$.
\end{proof}

Define the admissible class 
$$
\mathcal A_k = \{ \{\phi_i\}_{i=1}^k \colon 
\phi_i\colon V \rightarrow  [0,1]
\text{ and } 
\sum_{i=1}^k \phi_i  = 1 \}.
$$
Observe that the set of indicator functions for any $k$-partition of the vertices  is a member of $\mathcal A_k$. 
For 
$\{\phi_i\}_{i=1}^k \in \mathcal A_k$ and $\alpha> 0$, we define the \emph{relaxed energy}, 
$\Lambda_{k}^{\alpha}( \{\phi_\ell\}_{\ell=1}^k ) = \sum_{i=1}^{k}\lambda^{\alpha}(\phi_{i})$, 
where $\lambda^\alpha$ is defined in \eqref{eq:relDL}.
Thus, a relaxed version of the graph partitioning problem \eqref{eq:DirObjFun} can be formulated
\begin{equation}
\label{eq:relaxDirPart}
\Lambda_{k}^{\alpha,*} := \min_{ \{\phi_i\}_{i=1}^k \in \mathcal A_k } \   \sum_{i=1}^k  \lambda^\alpha (\phi_i).
\end{equation}

It is a consequence of Lemma \ref{lem:Conv} that for any $\{\phi_i\}_{i=1}^k \in \mathcal A_k$,  
$\Lambda_{k}^{\alpha}( \{\phi_i\}_{i=1}^k )$ is monotonically increasing in $\alpha$ and for any partition $V = \amalg_{i=1}^{k} V_{i}$, 
$\displaystyle\lim_{\alpha\rightarrow\infty} \Lambda_k^{\alpha}(\{\chi_{V_i}\}_{i=1}^{k}) = \Lambda_k(\amalg_{i=1}^{k}V_{i})$.
However, in practice, we desire a solution to \eqref{eq:relaxDirPart} for finite  $\alpha>0$.  We observe that $\Lambda_k^{\alpha}$ is bounded below by zero, and is being minimized over the compact set $\mathcal A_k$.  Thus a minimizer always exists. Supposing momentarily that we are able to find it, it is not yet clear how to interpret the collection $\{ \phi_i^* \}$, which attains the minimum, as a vertex partition, as sought in \eqref{eq:DirObjFun}. The following theorem, which is analogous to a continuous version in \cite[Thm. 2.3]{Bourdin2010}, tells us how this is accomplished.

\begin{theorem}  \label{thm:BangBang}
Let $k\in \mathbb Z^+$ and $\alpha>0$ be fixed. Every (local) minimizer of $\Lambda_k^{\alpha}$ over $\mathcal A_k$ is a collection of indicator functions.
\end{theorem}

To prove Theorem \ref{thm:BangBang}, we first prove the following lemma. 
\begin{lemma} \label{lem:lamConvex} For $\alpha>0$ fixed, $\lambda^{\alpha}(\phi)$ is a concave function of $\phi$.
\end{lemma}
\begin{proof}
Let $t\in(0,1)$ and $\phi_i\colon V\to \mathbb R$ for $i=1,2$. Using \eqref{eq:relDL} and the fact that the minimum is achieved for some normalized $\psi$, we compute
\begin{align*}
 \lambda^{\alpha}(t\phi_{1} + (1-t)\phi_{2}) &= \|\nabla\psi\|^{2}_{w,E} + \alpha \|\psi\|^{2}_{(1-t\phi_{1} - (1-t)\phi_{2})} \\
 &=\|\nabla\psi\|^{2}_{w,E} + t \alpha \|\psi\|^{2}_{(1-\phi_{1})} + (1-t) \alpha \|\psi\|^{2}_{(1-\phi_{2})} \\
 &\geq t\lambda^{\alpha}(\phi_{1}) + (1-t)\lambda^{\alpha}(\phi_{2} ).
\end{align*}
\end{proof}
\begin{proof}[Proof of Theorem \ref{thm:BangBang}.] Our proof closely follows the proof of \cite[Thm. 2.3]{Bourdin2010}. 
The set $\mathcal A_k$ is  the probability simplex in $\mathbb{R}^{k}$, and its extreme points are clearly given by the indicator functions.  As Lemma \ref{lem:lamConvex} shows, 
$\Lambda^{\alpha}$ is a concave function on $\mathcal A_k$, so has a minimum and at least one minimizer is an extreme point of $\mathcal A_k$.

Now suppose that there exists some $\{ \phi_i\}_{i=1}^k \in \mathcal A_k$ that achieves the minimum and is not an extreme point.  Since $\sum_{i=1}^k \phi_i = 1$ there exist at least two  $\phi$'s which are not $\{0,1\}$-valued at a vertex $v\in V$.  After re-indexing, suppose these are given by $\phi_{1}$ and $\phi_{2}$. Thus, there exists $\epsilon >0$ such that 
$\epsilon < \phi_{i}(v) < 1-\epsilon \quad i=1,2$.
By concavity of $\lambda^{\alpha}$, we have
\begin{subequations}
\label{eq:lamSums}
\begin{align}
\lambda^{\alpha}(\phi_{1}) \geq \frac{1}{2}\lambda^{\alpha}(\phi_{1} + \epsilon1_{v}) + \frac{1}{2}\lambda^{\alpha}(\phi_{1} - \epsilon1_{v}) \\
\lambda^{\alpha}(\phi_{2}) \geq \frac{1}{2}\lambda^{\alpha}(\phi_{2} + \epsilon1_{v}) + \frac{1}{2}\lambda^{\alpha}(\phi_{2} - \epsilon1_{v})
\end{align}
\end{subequations}
Adding these, and recognizing the right-hand side as an average, we must have
\begin{center}
$\lambda^{\alpha}(\phi_{1}) + \lambda^{\alpha}(\phi_{2}) \geq \min\{\lambda^{\alpha}(\phi_{1} + \epsilon1_{v}) +\lambda^{\alpha}(\phi_{2} - \epsilon1_{v}), \lambda^{\alpha}(\phi_{1} - \epsilon1_{v}) +\lambda^{\alpha}(\phi_{2} + \epsilon1_{v})\}$
\end{center}
But both terms in the minimum are feasible perturbations, thus by optimality of $\{\phi_i\}$, we must have equality:
\begin{center}
$\lambda^{\alpha}(\phi_{1}) + \lambda^{\alpha}(\phi_{2}) = \lambda^{\alpha}(\phi_{1} + \epsilon1_{v}) +\lambda^{\alpha}(\phi_{2} - \epsilon1_{v}) = \lambda^{\alpha}(\phi_{1} - \epsilon1_{v}) +\lambda^{\alpha}(\phi_{2} + \epsilon1_{v})$
\end{center}
But this implies equality in \eqref{eq:lamSums} as well:
\begin{center}
$\lambda^{\alpha}(\phi_{1}) = \frac{1}{2}\lambda^{\alpha}(\phi_{1} + \epsilon1_{v}) + \frac{1}{2}\lambda^{\alpha}(\phi_{1} - \epsilon1_{v})$\ \\
\ \\
$\lambda^{\alpha}(\phi_{2}) = \frac{1}{2}\lambda^{\alpha}(\phi_{2} + \epsilon1_{v}) + \frac{1}{2}\lambda^{\alpha}(\phi_{2} - \epsilon1_{v})$
\end{center}
\ \\
From the proof of Lemma \ref{lem:lamConvex}, we conclude that the eigenvector $\psi$ corresponding to $\phi_{1}$ is also an eigenvector for $\phi_{1} + \epsilon1_{v}$ and $\phi_{1} - \epsilon1_{v}$.  
We can subtract the following equations 
\begin{displaymath}
   \left\{
     \begin{array}{lr}
       \Delta_{G}\psi + \alpha(1- \phi_{1} - \epsilon1_{v})\psi = \lambda^{\alpha}(\phi_{1} + \epsilon1_{v})\psi \\
      \Delta_{G}\psi + \alpha(1- \phi_{1} + \epsilon1_{v})\psi = \lambda^{\alpha}(\phi_{1} - \epsilon1_{v})\psi
     \end{array}
   \right.
\end{displaymath} 
and using $\psi > 0$, simplify to yield
\begin{center}
$\phi_{1} + \epsilon 1_{v} - (\phi_{1} - \epsilon 1_{v}) = 2\epsilon 1_{v} \equiv C > 0$
\end{center}
for some constant $C$, which is clearly a contradiction.
\end{proof}

For fixed $\alpha>0$, we now consider the problem of solving the relaxed partitioning problem  \eqref{eq:relaxDirPart}. Since $\Lambda^\alpha_k\colon \mathcal A_k \to \mathbb R$ is Fr\'echet differentiable, we could apply a gradient descent algorithm analogous to the continuous method proposed in \cite{Bourdin2010}.  
Instead, we propose a \emph{rearrangement algorithm} (Algorithm \ref{alg:Rearrangement}). 
This is illustrated in Fig. \ref{fig:Rearr}. In Lemma \ref{lem:RearrDec},  we prove that Algorithm \ref{alg:Rearrangement} strictly decreases $\Lambda^{\alpha}_k$ at each iteration. This result is then strengthened in Theorem \ref{thm:rearr}, to show that not only do the iterates decrease the objective function, but the iterates terminate in a finite number of steps to a local minimum.  

\begin{algorithm}[t]
\caption{\label{alg:Rearrangement} A rearrangement algorithm for \eqref{eq:relaxDirPart}.}
\vspace{.2cm}
\begin{algorithmic}
\STATE{\bfseries Input:} An initial $\{ \phi_i\}_{i=1}^k \in \mathcal A_k$. \\

\vspace{.2cm}

\WHILE { not converged,}
\STATE For $i=1,\ldots,k$, compute the (positive and  normalized) eigenfunction $\psi_i$ corresponding to $\lambda^{\alpha}(\phi_i)$ in \eqref{eq:relDL}.
\STATE Assign each node $v\in V$ the label $i = \arg\max_j \ \psi_j(v) $.
\STATE Let $\{ \phi_i \}_{i=1}^k$ be the indicator functions for the labels.
\ENDWHILE
\end{algorithmic}
\end{algorithm}

\begin{lemma} \label{lem:RearrDec}
Assume $\{ \phi_{i} \}_{i=1}^k \in \mathcal A_k$ is not fixed by the rearrangement algorithm (Algorithm  \ref{alg:Rearrangement}).  Then one iteration of the rearrangement algorithm results in a strict decrease in $\Lambda^{\alpha}_k$.
\end{lemma}
\begin{proof}
Suppose $\{ \phi_{i} \}_{i=1}^k \in \mathcal A_k$ is not fixed by the rearrangement algorithm and let 
  $\{ \phi_{i}^+ \}_{i=1}^k \in \mathcal A_k$ be the next iterate. 
Let $\psi_i$ denote the first (normalized, positive) eigenvector of the operator $\Delta + \alpha(1-\phi_{i})$. We compute
\begin{subequations}
\label{eq:Edec}
\begin{align}
\Lambda^\alpha(\{\phi_i\}_{i=1}^{k}) &= \sum\lambda^{\alpha}(\phi_{i}) 
= \sum \| \nabla \psi_i \|_{2}^{2} + \alpha \|\psi_i \|_{(1-\phi_i)}^{2}  \\
\label{eq:Edecb}
& \geq\sum \|\nabla\psi_i  \|_{2}^{2} +\alpha \|\psi_i \|_{(1-\phi^+_i)}^{2} \\
\label{eq:Edecc}
& \geq \sum\lambda^{\alpha}(\phi^+_i). 
\end{align}
\end{subequations}
The inequality in \eqref{eq:Edecb} follows from the construction of the algorithm.
The inequality in \eqref{eq:Edecc} follows from  \eqref{eq:relDL}. Moreover, equality in \eqref{eq:Edecc} holds if and only if $\psi_i$ is also an eigenvector for the updated Schr\"dinger operator 
 $\Delta + \alpha(1-\phi^+_i)$ for all $i=1,\ldots,k$.  
 From the proof of Theorem \ref{thm:BangBang}, we find that $\phi_i - \phi^+_i$ is a constant function for all $i=1,\ldots,k$.  This contradicts the assumption that $\{ \phi_{i} \}_{i=1}^k$ is not fixed by the rearrangement algorithm. 
\end{proof}
\begin{theorem} \label{thm:rearr}
Let $\alpha>0$. For any initialization, the rearrangement algorithm \ref{alg:Rearrangement} terminates in a finite number of steps at a local minimum of  $\Lambda_k^\alpha$, as defined in \eqref{eq:relaxDirPart}. 
\end{theorem}
\begin{proof}
It follows from Lemma \ref{lem:RearrDec} and the finiteness of $V$ that for any initialization,  the rearrangement algorithm \ref{alg:Rearrangement} converges to a fixed point in a finite number of iterations. Thus, it suffices to show that every fixed point of the algorithm is locally optimal.
  Let $\{\phi_{i}\}_{i=1}^k$ be a fixed point of the rearrangement algorithm and let 
$\psi_i$ denote the first (normalized, positive) eigenvector of the operator $\Delta + \alpha(1-\phi_{i})$.  The Fr\'echet derivative of $\Lambda_k^\alpha \colon \mathcal{A}^{k} \to \mathbb R$ in the direction $\{ \delta \phi_i \}_{i=1}^k$ is written
\begin{align}
\label{eq:Frechet}
\Big\langle \frac{\delta \Lambda_k^\alpha}{\delta \{ \phi \}} , \{ \delta \phi_i \} \Big\rangle 
= -\alpha \sum_{i} \langle \psi_i^2, \delta\phi_{i} \rangle. 
\end{align}
For $\{\phi_{i}\}_{i=1}^k \in \mathcal A_k$, any admissible perturbation can be written 
$$
\tilde{\phi_{i}} = \phi_{i} + \sum_{v\in V} t_{i,v}\chi_{\{v\}}, \qquad i=1,\ldots,k
$$
for constants $t_{i,v}$, such that for every $v\in V, \ \sum_{i=1}^k t_{i,v} = 0$ and
$$
t_{i,v} \begin{cases} \geq 0 &  \text{if }\phi_{i}(v) = 0 \\
\leq 0 & \text{if } \phi_{i}(v) = 1.
\end{cases} 
$$
 Using \eqref{eq:Frechet}, we compute 
 \begin{align*}
 \Big\langle \frac{\delta \Lambda_k^\alpha}{\delta \{ \phi \}} , \{ \delta \phi_i \} \Big\rangle 
= -\alpha\displaystyle\sum_{v}\sum_{i} t_{i,v} \psi_i^2(v) 
  \geq -\alpha\displaystyle\sum_{v}\sum_{i} t_{i,v}\psi^2_{i^{*}(v)}(v) = 0
\end{align*}
where $i^{*}(v) = \arg\max_{i}\psi_{i}(v)$.  This proves local optimality.  
\end{proof}

\begin{figure}[t!]
\begin{center}
\includegraphics[height=6cm]{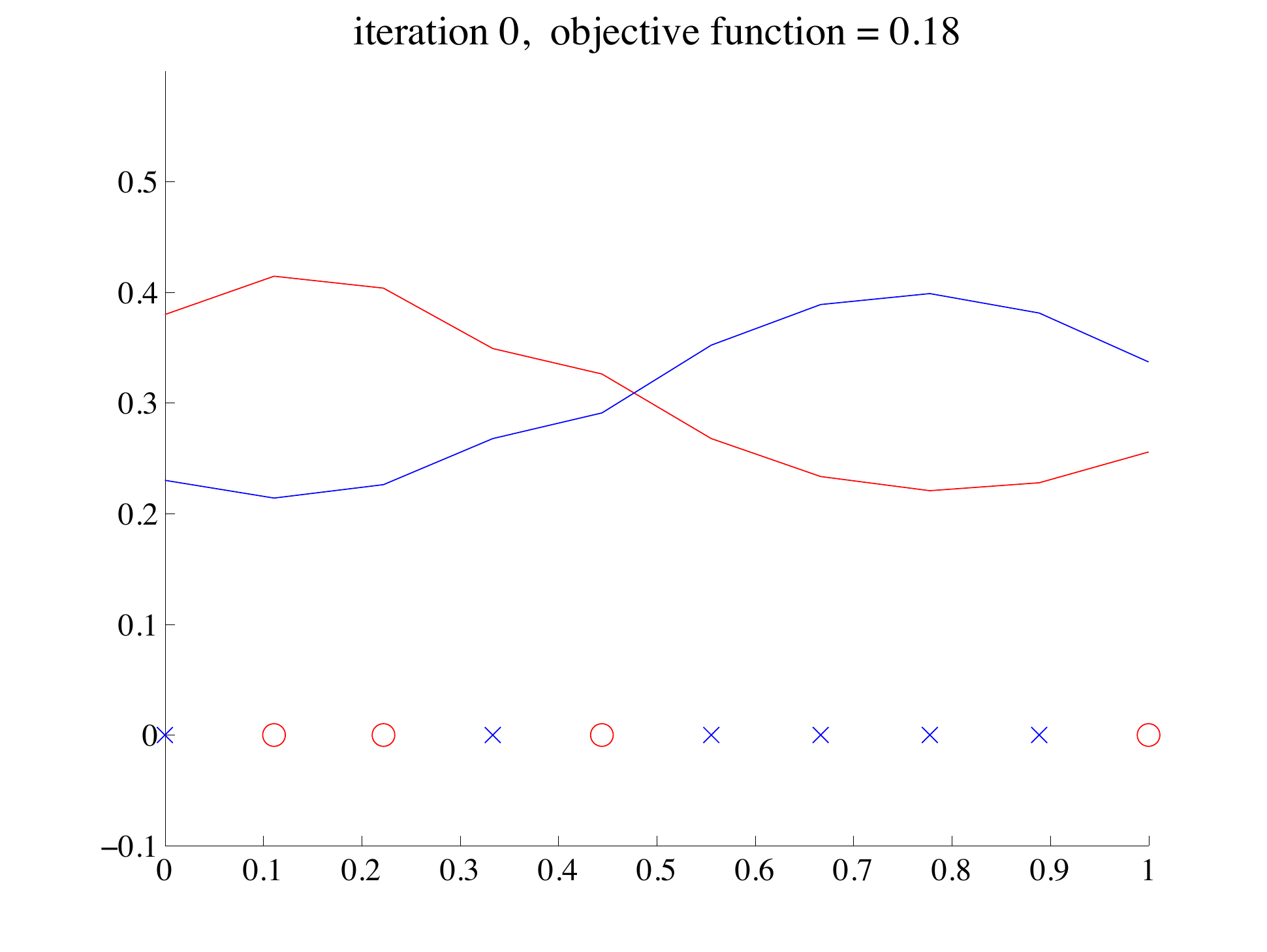}
\includegraphics[height=6cm]{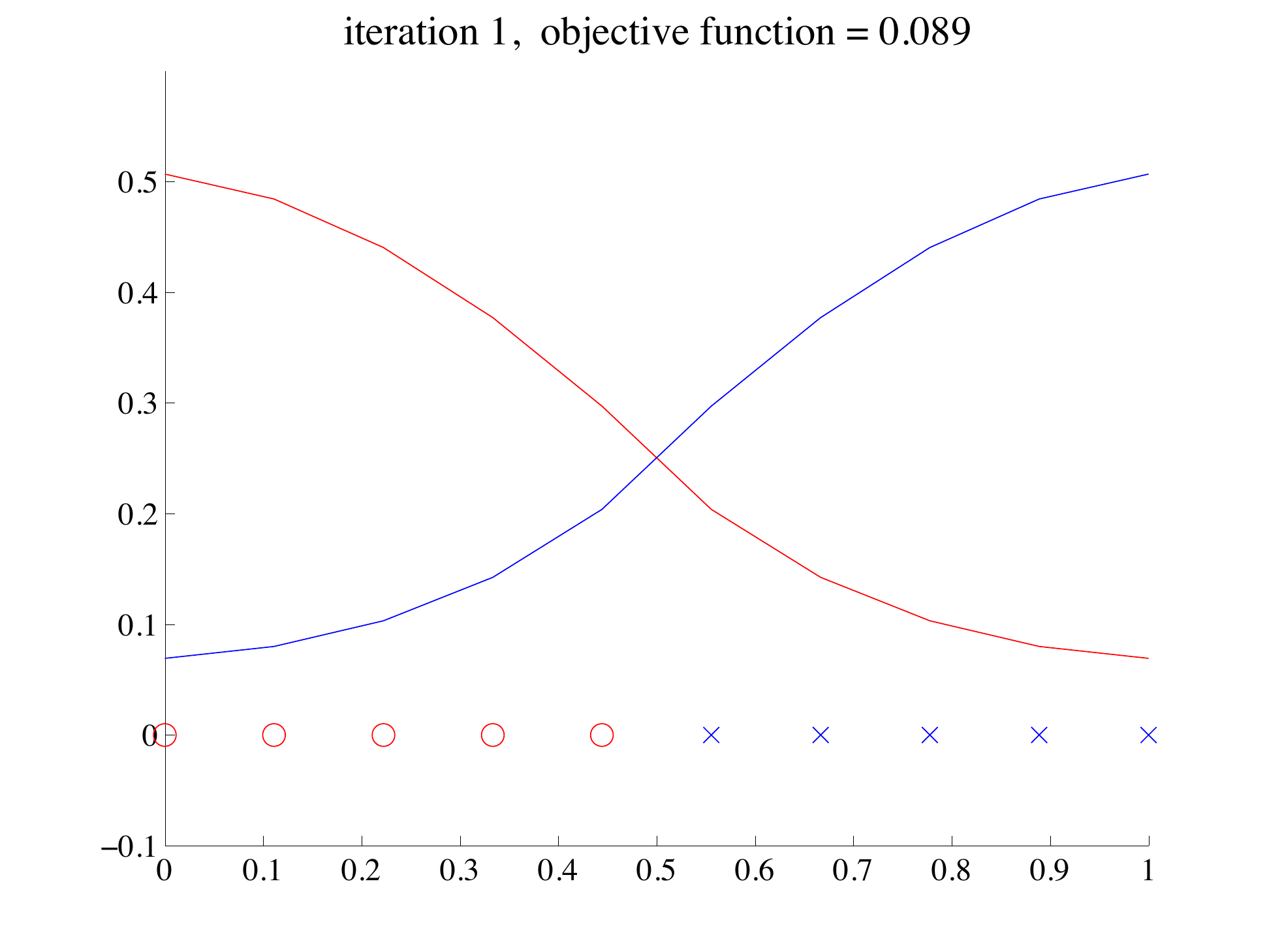}
\caption{An illustration of the rearrangement algorithm \ref{alg:Rearrangement} for $k=2$ on a  line graph with $n=10$ vertices,  an equally weighted nearest-neighbor similarity matrix, and graph Laplacian with $r=0$.
{\bf (left)}  The algorithm is randomly initialized as indicated by the blue \texttt{x} and red \texttt{o}. The eigenfunctions corresponding to the potential given by the indicator function on each partition is also indicated in the same color. 
{\bf (right)} The algorithm converges to the globally optimal stationary state on the right in one iteration. See \S \ref{sec:rel}.}
\label{fig:Rearr}
\end{center}
\end{figure}

\begin{remark} We refer to algorithm \ref{alg:Rearrangement} as a rearrangement algorithm since at each iteration, the vertex functions $\{ \phi_i \}$ are rearranged to  decrease \eqref{eq:relaxDirPart}. These types of methods were introduced by  Schwarz and  Steiner and have wide applications in variational problems \cite{kawohl2000}. For example, Steiner rearrangement can be used to prove the isoperimetric inequality that the  ball is the minimal perimeter domain amongst all regions of equal measure. More recently, rearrangement algorithms have been used in eigenvalue optimization problems  including Krein's problem: Given an open, bounded connected domain $\Omega \subset \mathbb R^2$ and a prescribed amount of two materials of different density, find the distribution which minimizes the smallest frequency of the clamped drum \cite{Cox1991,Chanillo2000,Kao2008,Kao2012}.  

Algorithm \ref{alg:Rearrangement}  also shares many attributes with the Merriman, Bence, and Osher (MBO) algorithm for approximating the motion by mean curvature \cite{MBO1992,MBO1993,mbo2012,garcia2013,esedoglu2013,Gennip2013}.   
\end{remark}

\subsection{Advice on $\alpha$ and $r$}\label{alpha_ad}

All machine learning algorithms require user input in various forms.  For the rearrangement algorithm to find a meaningful optimum, the parameter $\alpha$ must be chosen carefully.  Lemma \ref{lem:Conv} might suggest to the reader that taking $\alpha$ large is a good idea; however, this is not the case.  Large values of $\alpha$ force the associated eigenvectors to localize immediately, and the rearrangement algorithm terminates.  We make this precise with Lemma \ref{lem:big_alpha} below.  Moreover, comparing partition quality for different values of $\alpha$ is not possible; as was noted earlier, a consequence of the proof of Lemma \ref{lem:Conv} is that the relaxed energy is monotonic in $\alpha$. 

\begin{lemma}\label{lem:big_alpha} Consider any partition of $V$ consisting entirely of connected sets, $V = \displaystyle\amalg_{i=1}^{k}V_{i}$.  Then there exists $\alpha$ sufficiently large so that this partition is locally optimal for $\Lambda_{k}^{\alpha,*}$.
\end{lemma}

We first require the following lemma, which has a well-known analogue in the continuous setting.  

\begin{lemma}\label{lem:d_pos}
For any connected $S\subseteq V$, the first Dirichlet eigenvector, $\psi^{D}(S)$, attaining the minimum in \eqref{eq:lam}, is strictly positive on $S$.
\end{lemma}

\begin{proof} We assume $r=0$ so that $\Delta = D - W$ in this calculation, but the result holds for all $r\in[0,1]$ by applying a transformation to the eigenvector which preserves positivity. 
By Lemma \ref{lem:Conv},  $\psi^{\alpha} \rightarrow \psi^{D}$.  Moreover, using the Perron-Frobenius Theorem we see that $\psi^{D}\geq 0$.  Assume $\psi^{D}(i) = 0$ for some $i\in S$.  Then by definition we have
$$
\Delta\psi^{D}(i) =\lambda^{D}(S)\psi^{D}(i) = 0 \qquad  
\implies \qquad   \sum_{j} w_{ij} (- \psi^{D}(j)) = 0. 
$$
Since the left-hand side summation consists entirely of non-positive terms and using the connectedness of $S$, we must have at least one $w_{ij} > 0$ implying $\psi^{D}(j) = 0 $;  iterating this argument implies $\psi^{D} \equiv 0$, which is a contradiction.
\end{proof}

\begin{proof}[Proof of Lemma \ref{lem:big_alpha}] Now, consider any partition of $V$ consisting entirely of connected sets, $V = \displaystyle\amalg_{i=1}^{k}V_{i}$.  After possibly passing to a subsequence we know $\psi^{\alpha}_i \rightarrow \psi^{D}_i$ pointwise, and thus we can choose $\alpha$ large enough so that $\psi^{\alpha}_i(v) = \max_j\{\psi^{\alpha}_j(v)\}$ for all $v\in V_i$ and for all $i$.  This implies that this partition results in a stationary point of the rearrangement algorithm for this choice of $\alpha$ large, which proves local optimality by Theorem \ref{thm:rearr}.
\end{proof}

Conversely, taking $\alpha$ too small also has undesirable consequences.  We find experimentally that for $\alpha$ chosen too small, partitions can be lost.  By considering the decomposition 
$$
\Delta + \alpha (1-\chi_S) = \sum_i \lambda_i \psi_i \psi_i^t + \alpha \sum_{i\in S} \delta_i \delta_i^t,
$$
where $\delta_i$ indicates the Kronecker-delta function on vertex $i$, 
a convincing heuristic is to pick $\alpha$ on the scale of the eigenvalues of the graph Laplacian.  Following this line of thought, in most of our numerical examples (see \S\ref{sec:num}) we pick $\alpha$ to be approximately $k\lambda_{2}$, where $\lambda_{2}$ is the second smallest eigenvalue of the graph Laplacian.  We hope to make this heuristic more precise in future work.

Another choice lies in determining the parameter $r$ in the definition of the graph Laplacian, $\Delta_r = D^{-r}(D-W)$; see \eqref{eq:explcitEigEq}. We note that, as in spectral clustering, the choice of $r$ determines the inner product on the vertices which in turn defines the volume of vertex subsets; $r=0$ corresponds to cardinality whereas $r=1$ corresponds to degree-weighted volume. 
For the former case, spectral clustering heuristically minimizes RatioCut yielding approximately equal cardinality components and for the latter case, spectral clustering heuristically minimizes NCut yielding approximately equal volume components \cite{Luxburg:2007}.  We expect and have observed from numerical experiments that the same intuition applies for our model.

\subsection{A semi-supervised extension}\label{semi} 

In many clustering applications, a small percentage of the data labels are known and thus it is desirable for a clustering algorithm to have a \emph{semi-supervised extension} that allows for the incorporation of such information. The rearrangement algorithm \ref{alg:Rearrangement} 
 has a natural semi-supervised variant. The label membership of a subset of the points can be fixed in the algorithm and the reader may check that all proofs of convergence remain valid. 
Moreover, fixing these points will force the eigenvectors to `spread out' accordingly.  We apply this variant to the MNIST handwritten digit data in Section \ref{sec:mnist}.

\section{Numerical experiments} \label{sec:num}
In this section, we apply our graph partitioning method to solve \eqref{eq:relaxDirPart} for several example graphs. The rearrangement algorithm requires the computation of the ground state of the eigensystem  \eqref{eq:RelaxedExplicitEig}. 
In \eqref{eq:RelaxedExplicitEig}, the eigenvector $\psi$ is normalized according to $\psi^t D^r \psi = 1$.  A standard eigenvalue problem can be realized by the change of variables, $\eta = D^{r/2} \psi$, and is written 
$$
 [D^{1-r} - D^{-r/2} W D^{-r/2} + \alpha (1-\phi)  ] \  \eta = \lambda \  \eta, \qquad \eta^t \eta = 1.
$$ 
The rearrangement algorithm (Algorithm \ref{alg:Rearrangement}) was implemented in Matlab 7.12 using sparse matrices and the \verb+eigs+ command.  In each case, the Ritz estimate residual convergence criteria was set using the Matlab option $\verb+tol+ = 10^{-4}$.

\begin{figure}[t!]
\begin{center}
\includegraphics[height=7.2cm]{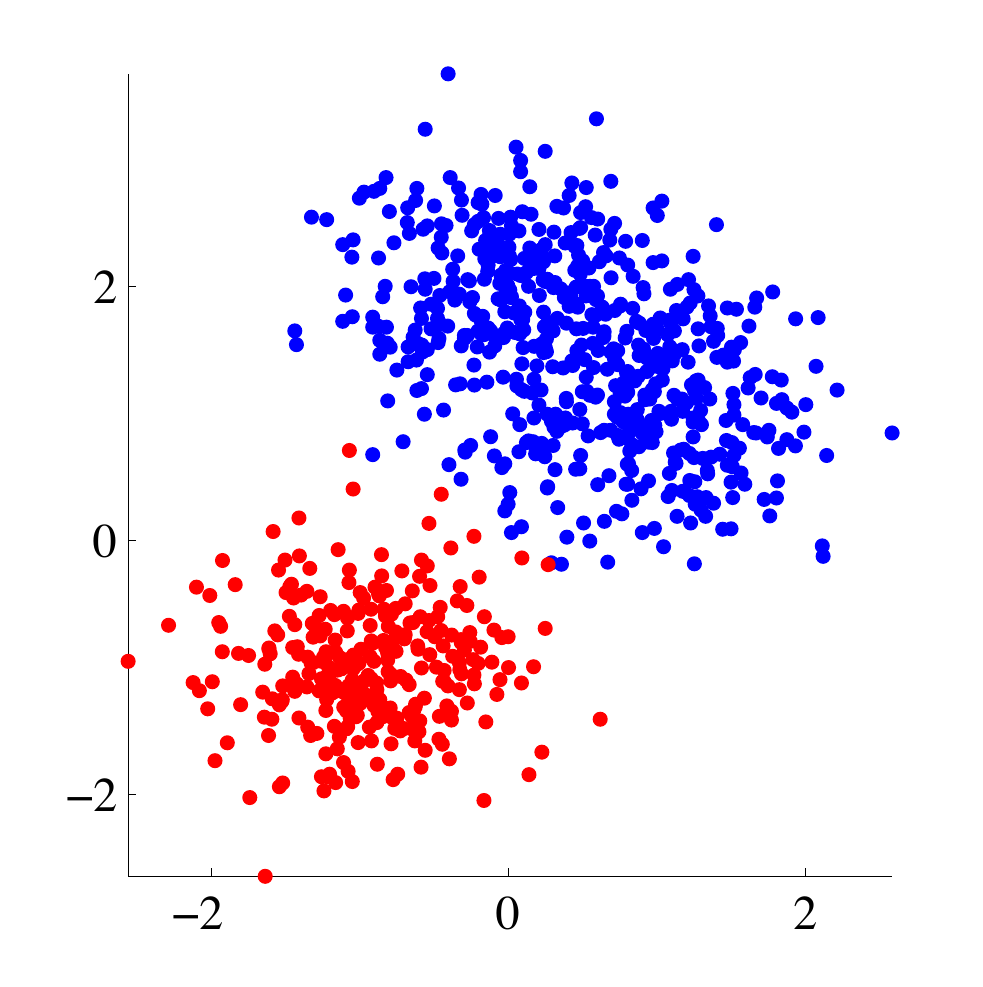}
\includegraphics[height=7.2cm]{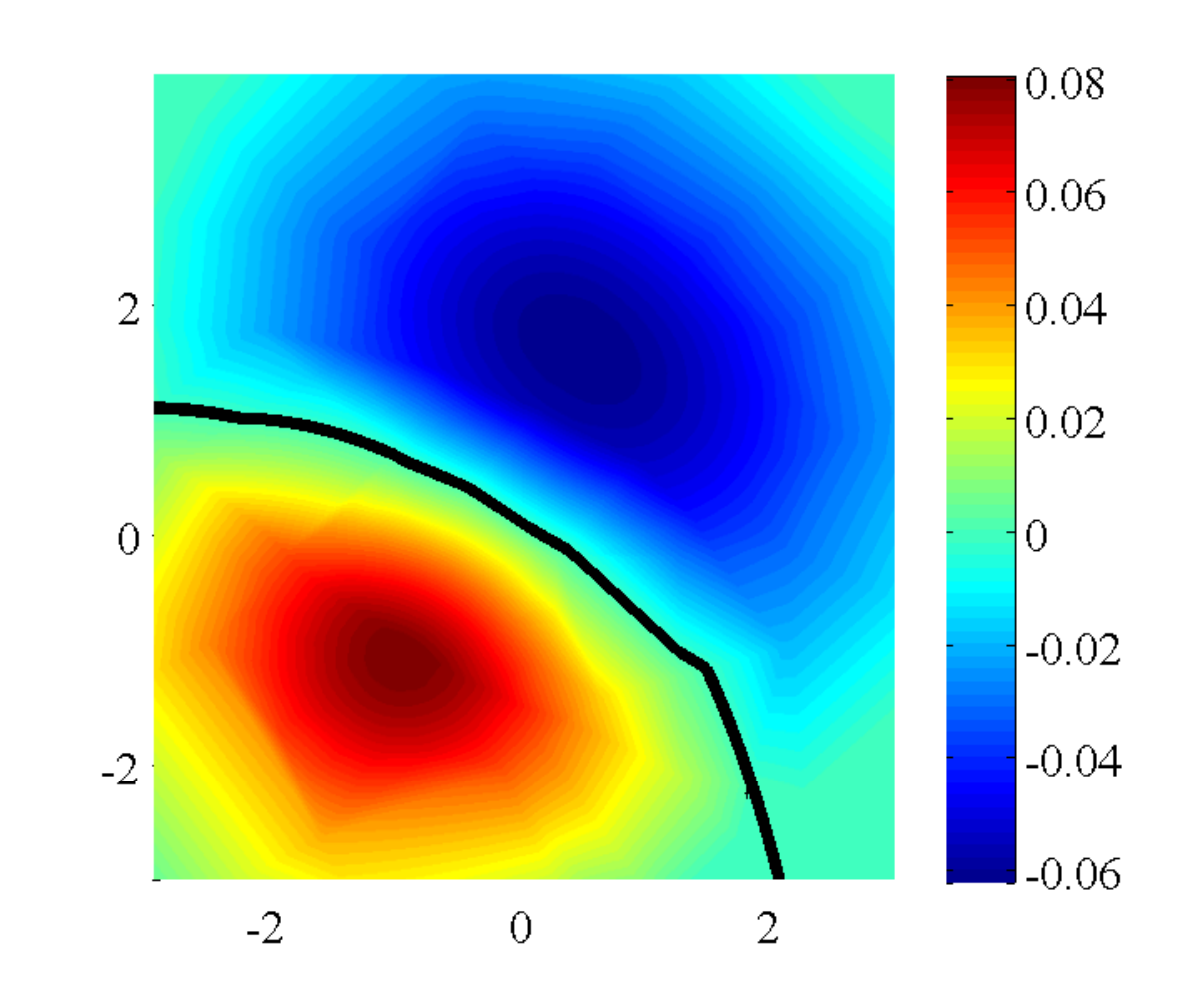}
\caption{{\bf (left)} The outputted partition of the rearrangement algorithm \ref{alg:Rearrangement}. 
{\bf (right)} An interpolated contour plot of $\psi_{1} -\psi_{2}$ in the feature space.  The black line represents the decision boundary for clustering a new datapoint. See \S \ref{sec:GMM}.}
\label{fig:rear}
\end{center}
\end{figure}

\subsection{Gaussian mixture model} \label{sec:GMM}
As a simple illustration, we first consider a Gaussian Mixture Model with two Gaussian clouds of differing sizes as illustrated in the left panel of Figure \ref{fig:rear}.  To construct the similarity matrix, we used a standard Gaussian kernel, $e^{-d^{2}(x,y)/\sigma^{2}}$, with $\sigma = 1$.  We choose the graph Laplacian with $r=1$ for this experiment and the algorithm is randomly initialized.  
We choose $\alpha = 2\lambda_{2}$, where $\lambda_{2}$ is the second eigenvalue for the graph Laplacian. The algorithm converges to the global optimum in 5 iterations. 
In the left panel of Figure \ref{fig:rear} we plot the optimal partition. 

In the two-dimensional feature space for the model, we interpolate the values of the eigenvectors $\psi_i$ corresponding to each partition and plot the  difference $\psi_{1} - \psi_{2}$ in the right panel. 
The eigenvectors concentrate where the clusters are most dense; this is in contrast to the $r=0$ case, where  the eigenvectors tend to be flat.  The values of the eigenvectors can be interpreted as a confidence for the labeling of a particular vertex.  The two points given by $x^* = \arg\max\psi_{1}$ and $z^*=\arg\max\psi_{2}$ are thus ``good representatives'' for the two clusters in that they have the highest confidence.  Moreover, the  contour plot in the right panel of Figure \ref{fig:rear} illustrates the supervised extension of the algorithm discussed in \S\ref{semi}. Namely, if we wish to cluster a new datapoint $z\in \mathbb R^2$, we can use interpolated values for the eigenvectors to assign it to the cluster $i = \arg\max_{j}\psi_{j}(z)$. The decision boundary for this assignment is given by the black line in Figure \ref{fig:rear}.

\begin{figure}[t!]
\begin{center}
\includegraphics[width=8.1cm]{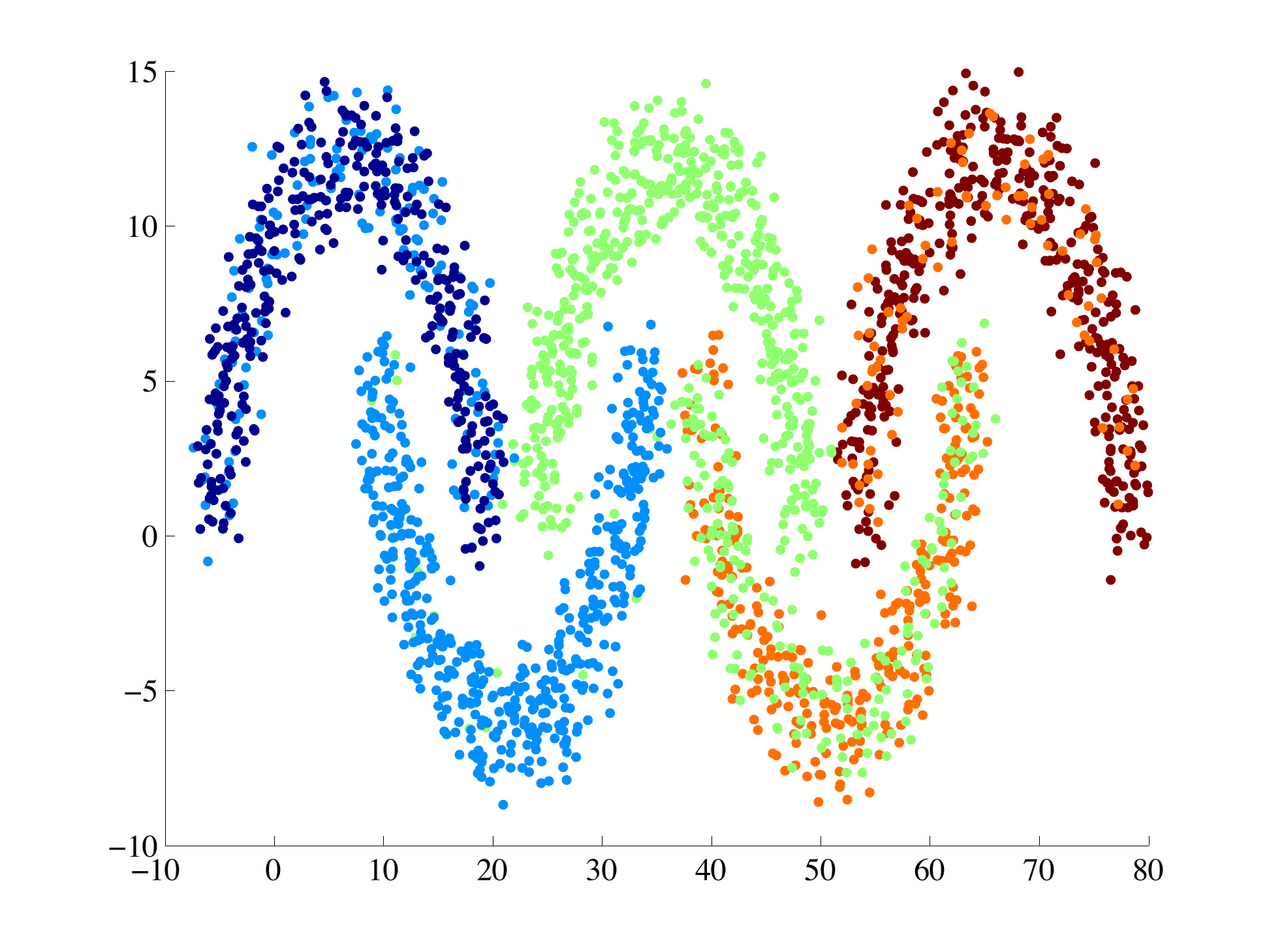}
\includegraphics[width=8.1cm]{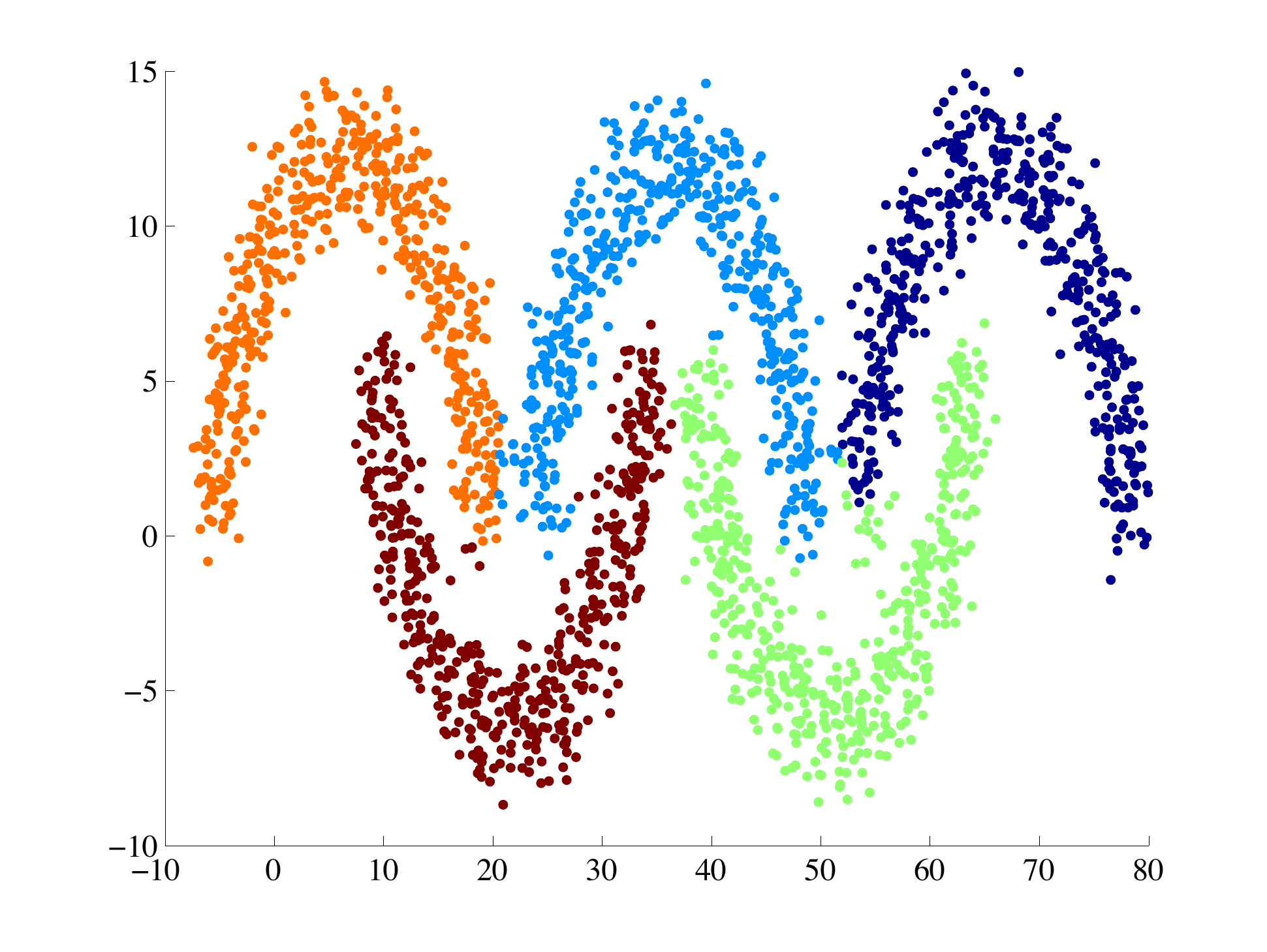}
\caption{{\bf (left)} The lowest energy partition of a five moons dataset obtained via the normalized spectral clustering method from 50 random initializations of $k$-means. 
{\bf (right)} The lowest energy partition of the same five moons dataset after ten iterations of the rearrangement algorithm from 10 random initializations.  The spectral clustering method has significantly lower purity than our method. See \S \ref{sec:3moons}.}
\label{fig:3moons}
\end{center}
\end{figure}

\subsection{Five  moons} \label{sec:3moons}
Next we consider the five moons partitioning problem. We construct a similarity matrix using a Gaussian kernel, $e^{-d^{2}(x,y)/\sigma^{2}}$, with $\sigma = 1$, and the Laplacian is constructed using $r=1$.  We choose $\alpha = 5\lambda_{2}$, where $\lambda_{2}$ is the second eigenvalue for the graph Laplacian. For ten different random initializations, we run the algorithm until convergence and choose the lowest energy partition. 
In Figure \ref{fig:3moons}(right), we give a scatter plot of the points where colors represent the  labels  assigned to each point by the algorithm. 
For eight of the ten initializations, the algorithm generated a partition very similar to that shown in Fig. \ref{fig:3moons}. For each initialization, the algorithm converges in approximately 10 iterations. 
As a comparison, we also use the popular normalized spectral clustering method \cite{Shi2000} to obtain the partition in Figure \ref{fig:3moons}(left). The partition in the figure is the  lowest energy partition obtained from 50 random initializations of $k$-means.

\subsection{Several small datasets} \label{sec:smallData}
We obtained the similarity matrices for twelve  small datasets from the website of Z. Yang  \cite{Yang2012,ZhirongYangWebpage},  to which we refer the reader for a complete description and source information. We  apply   the rearrangement algorithm to each dataset  using 20 random initializations with $\alpha = k \lambda_2$ and $r=0$. 
 In the following table, we report the size of each dataset $n$, the desired number of clusters $k$, the  purity corresponding to the lowest energy partition,  a comparison value for the purity, the average number of iterations required for convergence of the rearrangement algorithm, the smallest objective function value obtained, and the objective function value of the ground truth labels.  
 Comparison values for the purity are taken to be the best purity obtained by a comparison of ten different methods   \cite[Table 1]{Yang2012}. It should be noted that no single method obtained the comparison purity values.  
 We remind the reader that each iteration of the algorithm  involves the computation of  the ground state of $k$ $n\times n$ standard eigenvalue problems; for these relatively small scale problems the computational costs associated with this algorithm are minimal.  We observe that the found objective value is always smaller than the ground truth objective value. This demonstrates   that the non-convexity of the problem is \emph{not} preventing the algorithm from finding a good minimizer.

\begin{center}
\begin{tabular}{c |c c c c  c c c}
&&&&purity &avg. & found &ground truth \\
Dataset & $n$ & $k$ & purity & comp. & iterations &  obj. &  obj. \\
\hline
STRIKE & 24 & 3 & 0.96 & 1.00 & 3.9 & 0.362 & 0.367 \\ 
AMLALL & 38 & 3 & 0.92 & 0.92 & 4.5 & 1.68 & 1.732 \\ 
DUKE & 44 & 2 & 0.52 & 0.70 & 5.2 & 0.549 & 1.019 \\ 
KHAN & 83 & 4 & 0.57 & 0.60 & 4.8 & 1.37 & 2.160 \\ 
POLBOOKS & 105 & 3 & 0.81 & 0.83  & 6.0 & 0.975 & 1.291 \\ 
CANCER & 198 & 14 & 0.53 & 0.54  & 5.1 & 13.3 & 16.441 \\ 
SPECT & 267 & 3 & 0.79 & 0.79 & 8.8 & 0.768 & 1.759 \\ 
ROSETTA & 300 & 5 & 0.77 & 0.77 & 7.8 & 5.62 & 12.482 \\ 
ECOLI & 327 & 5 & 0.80 & 0.83 & 7.0 & 0.568 & 0.656 \\ 
IONOSPHERE & 351 & 2 & 0.70 & 0.70 & 6.5 & 0.119 & 0.205 \\ 
DIABETES & 768 & 2 & 0.65 & 0.65 & 3.0 & 0.00552 & 0.013 \\ 
ALPHADIGS & 1404 & 36 & 0.46 & 0.51 &  12.9 & 37.5 & 56.503 
\end{tabular}
\end{center}

\subsection{MNIST handwritten digits} \label{sec:mnist}
The MNIST handwritten digit dataset consists of 70,000 $28\times28$ greyscale images of handwritten digits 0 to 9.  
As input we used the similarity matrix for the MNIST dataset obtained from the website of Z. Yang \cite{Yang2012,ZhirongYangWebpage}. 
%As input, we used a standard similarity matrix construction for MNIST, obtained from the authors of \cite{Bresson2013}: let $\sigma$ be the average of the Euclidean distances of all the data points, and set $W_{ij} = e^{-d(x_j,x_i)/2\sigma^{2}}$ if datapoint $j$ is one of $i$'s 5 nearest neighbors, and zero otherwise.  
We symmetrize this matrix via $\tilde{W}_{ij} = \max\{W_{ij},W_{ji}\}$, take $r=0$, and set $\alpha = 10\lambda_{2}$.  Moreover, we randomly sampled 3\% of the data and used the semi-supervised variant of our algorithm (see \S \ref{semi}). The remaining initialized labels were assigned randomly.   

For ten different random initializations, we run the algorithm until convergence and choose the lowest energy partition. In each case, the algorithm converges in approximately 20 iterations. The purity obtained, as defined in  \cite{Yang2012}, is 0.961 which is comparable to the performance of state-of-the-art clustering algorithms. 
We note that the partitions identified for other initial configurations had  similar energy and purity values. 
Figure \ref{fig:mnist_fig} is a graphical display of the quality of the output.  On the left-hand side are the representative images for each cluster (where each eigenvector achieves its maximum), and on the right are the averaged images within each cluster. In general, the maximum value of the eigenvectors may be non-unique, but for this dataset, the maximum was unique. In Table \ref{table:nonlin}, we display the confusion matrix for the obtained clusters. The columns represent the ground truth labels and the rows represent the labels assigned by the proposed algorithm. Each column sums to one and represents the distribution of the true labels across our partitions.

\begin{figure}[t!]
\begin{center}
\includegraphics[height=6cm]{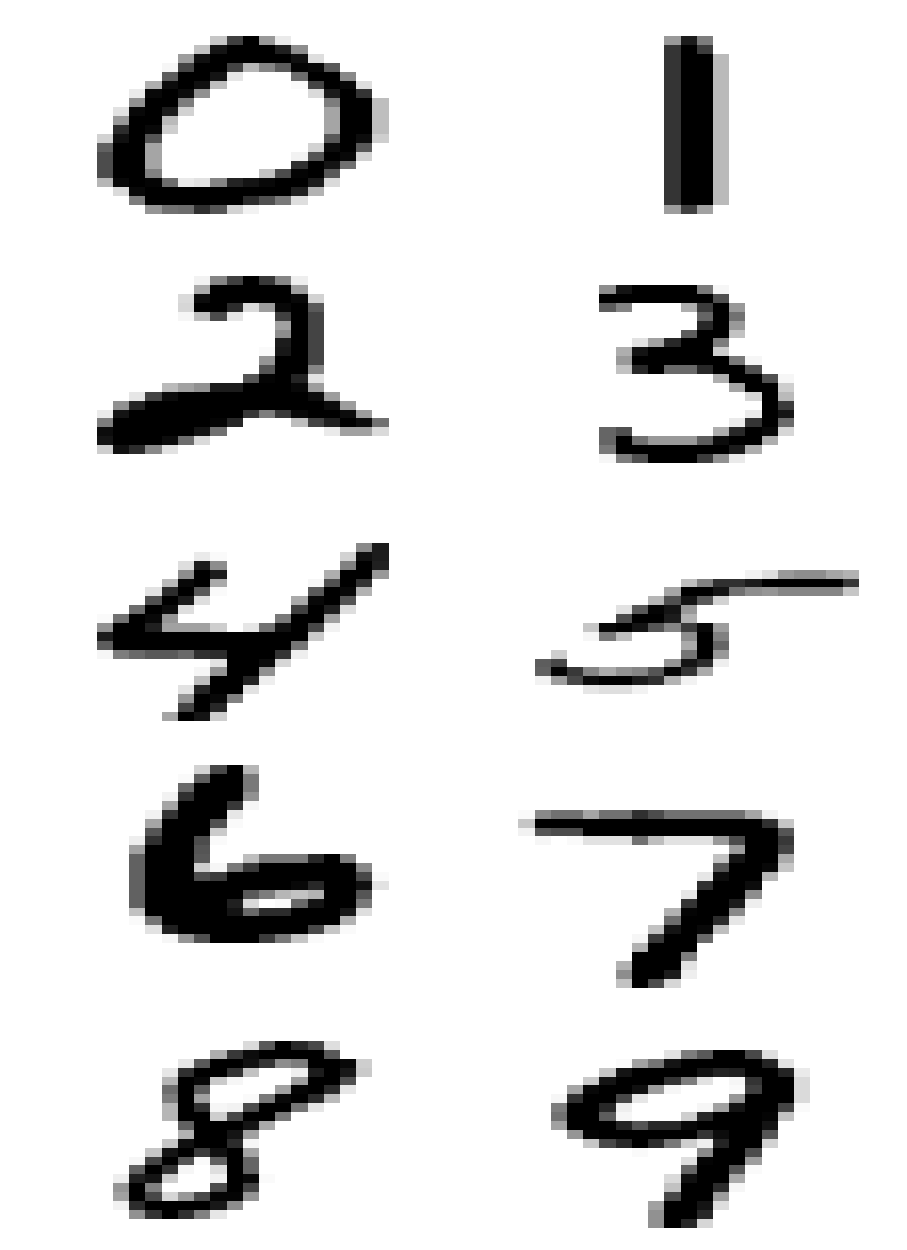} \qquad \qquad \qquad
\includegraphics[height=6cm]{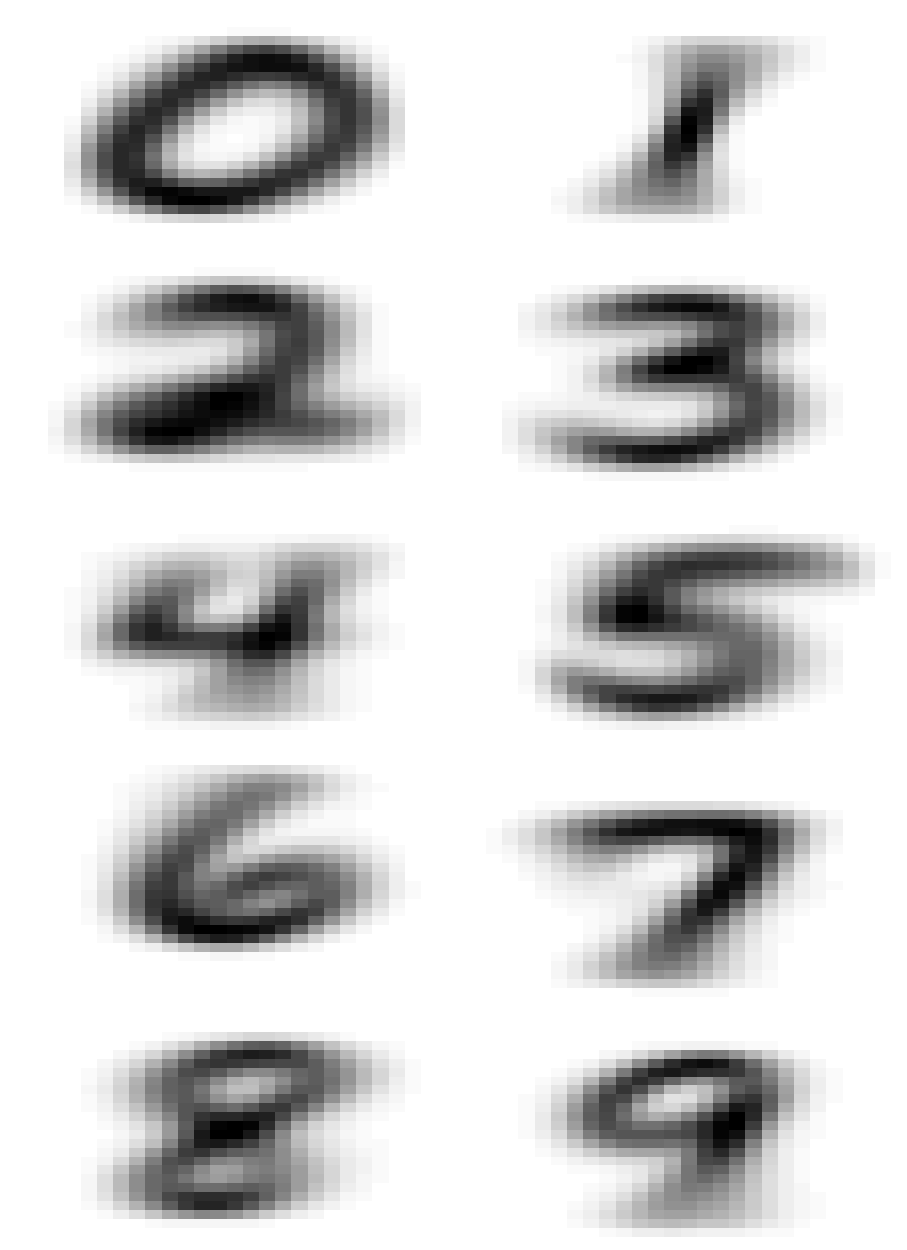}
\caption{MNIST handwritten digits; each image is $28\times28$ pixels. 
{\bf (left)} Representative images for each of the $k=10$ clusters. 
 {\bf (right)} The cluster means. See \S \ref{sec:mnist}.}
\label{fig:mnist_fig}
\end{center}
\end{figure}

\begin{table}[t!]
\centering % used for centering table
{\small 
\begin{tabular}{c | c c c c c c c c c c} % centered columns (10 columns)
\ & 0 & 1 & 2 & 3 & 4 & 5 & 6 & 7 & 8 & 9 \\
\hline 
0 & \textbf{0.9864} & 0.0147 & 0.0021 & 0.0008 & 0.0002 & 0.0006 & 0.0005 & 0.0009 & 0.0040 & 0.0000\\
1 & 0.0012 & \textbf{0.9717} & 0.0023 & 0.0006 & 0.0011 & 0.0012 & 0.0047 & 0.0046 & 0.0013 & 0.0012\\ 
2 & 0.0012 & 0.0047 & \textbf{0.9537} & 0.0000 & 0.0079 & 0.0001 & 0.0041 & 0.0076 & 0.0054 & 0.0003\\ 
3 & 0.0005 & 0.0004 & 0.0000 & \textbf{0.9894} & 0.0002 & 0.0045 & 0.0009 & 0.0001 & 0.0332 & 0.0001\\ 
4 & 0.0000 & 0.0003 & 0.0213 & 0.0002 & \textbf{0.9736} & 0.0060 & 0.0003 & 0.0076 & 0.0045 & 0.0010\\ 
5 & 0.0018 & 0.0001 & 0.0000 & 0.0005 & 0.0032 & \textbf{0.9811} & 0.0000 & 0.0029 & 0.0000 & 0.0009\\ 
6 & 0.0038 & 0.0065 & 0.0003 & 0.0027 & 0.0006 & 0.0000 & \textbf{0.9792} & 0.0009 & 0.0992 & 0.0001\\ 
7 & 0.0033 & 0.0008 & 0.0033 & 0.0029 & 0.0112 & 0.0035 & 0.0011 & \textbf{0.9675} & 0.0080 & 0.0014\\ 
8 & 0.0013 & 0.0003 & 0.0168 & 0.0027 & 0.0011 & 0.0004 & 0.0090 & 0.0054 & \textbf{0.8440} & 0.0017\\ 
9 & 0.0005 & 0.0004 & 0.0001 & 0.0003 & 0.0008 & 0.0026 & 0.0002 & 0.0024 & 0.0004 & \textbf{0.9932}\\
\end{tabular} } 
\label{table:nonlin} % is used to refer this table in the text
\caption{{\bf MNIST confusion matrix.} The columns represent the ground truth labels and the rows represent the labels assigned by the proposed algorithm. Each column sums to one and represents the distribution of the true labels across our partitions.  We see that the algorithm does very well, but confused some 8s for 6s. See \S \ref{sec:mnist}.} % title of Table
\end{table}

Finally, we explore the real-time computational costs of the rearrangement algorithm as well as the effect of using partially  labelled data. 
% Joshua16
% cat /proc/cpuinfo
% free -m 
 We  apply   the rearrangement algorithm to the MNIST dataset  using 10 random initializations with $\alpha = 10 \lambda_2$ and an increasing percentage of labelled data points. %The results are reported in the following table and can be  directly compared to  \cite[Section 5]{Bresson2013}. 
 For each percentage of labels, we report 
 the purity corresponding to the lowest energy partition, 
 the average  number of iterations (across  initializations) required for convergence of the rearrangement algorithm, 
 the average clock-time for convergence, and 
 the smallest objective function value obtained. 
 The objective function value for the ground truth labels is 1.6051.  We observe that the average convergence times are greater than those reported in \cite[Section 5]{Bresson2013}, however the Dirichlet energy partitions contain additional geometric information.  We also observe that for a typical random initialization, the purity is already 80-85\% by the second iteration of the algorithm. All computational times reported below  were obtained on a 2.67 GHz Intel Xeon  desktop computer with 48GB of RAM.

\begin{center}
{\bf MNIST dataset, $n=70,000$, $k=10$} 

\vspace{.2cm}

\begin{tabular}{r | c c c c}
labels & purity & avg. time (s) & avg. iterations  & found obj.   \\
\hline
% 0\% & & & & \\
 1\% & 0.8683 & 115.09 & 29.2 & 1.398\\
 2.5\% & 0.9619 & 51.15 & 11.4 & 1.563\\
 5\% &  0.9702 & 57.27 & 15.3 & 1.565 \\ 
  10\% & 0.9711 & 48.22 & 11.5 & 1.567 \\ 
\end{tabular}
\end{center}

\subsection{Torus} \label{sec:tori}
We consider a $120\times120$ square grid with periodic boundary conditions and 
construct the nearest-neighbor $r=0$ graph Laplacian. This is precisely the finite difference approximation of the Laplacian for the flat torus. 
We initialize the partition with a Voronoi tessellation for  randomly chosen generators.  In Fig. \ref{fig:tori}, we plot locally optimal partitions with the smallest  energy for $k=9$, 16, and 25.
We observe that the partitions are a hexagonal tiling as conjectured by  \cite{Cafferelli2007} and also computationally studied in \cite{Bourdin2010}.

\begin{figure}[t!]
\begin{center}
\includegraphics[width=5cm]{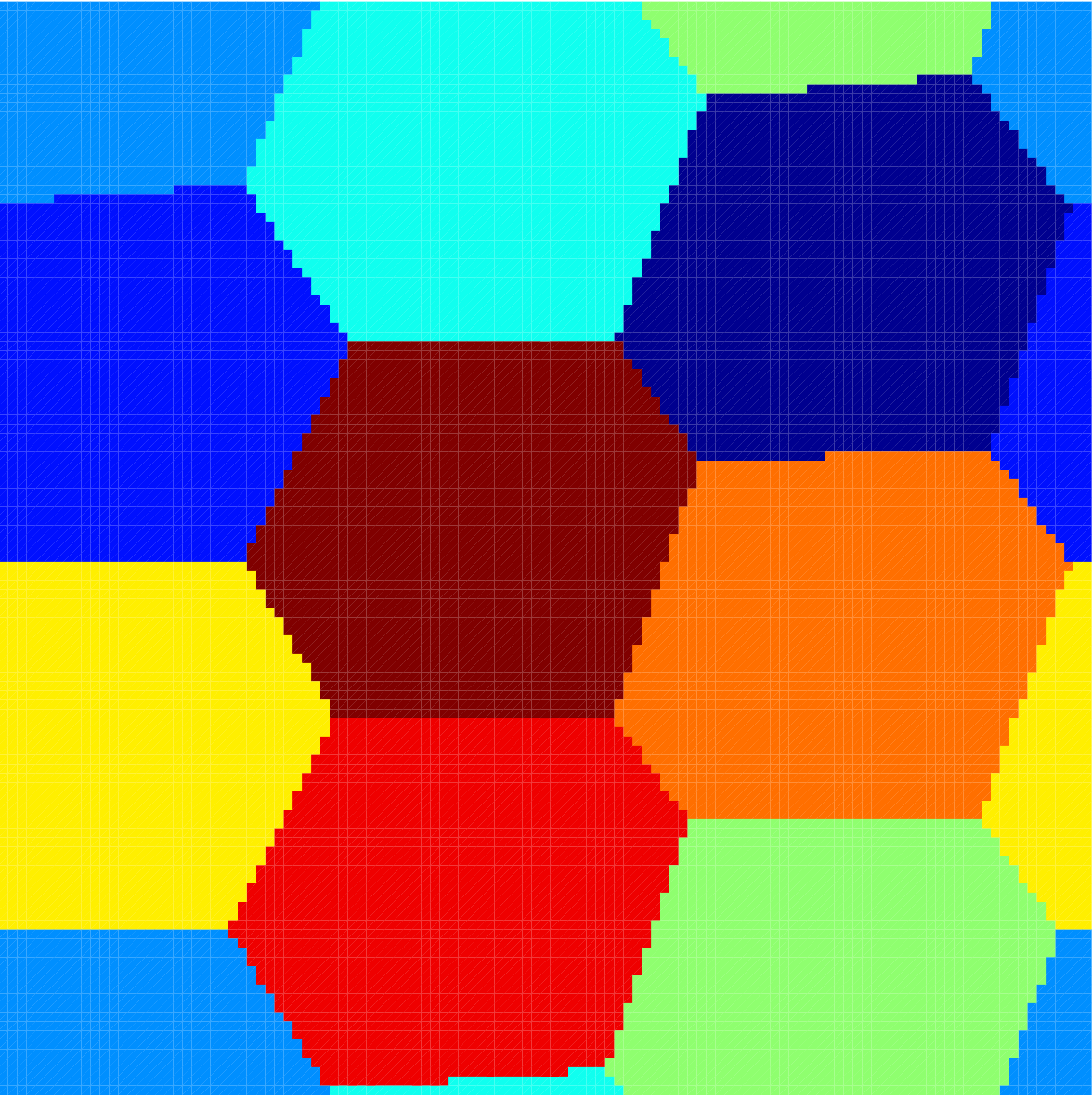}
\includegraphics[width=5cm]{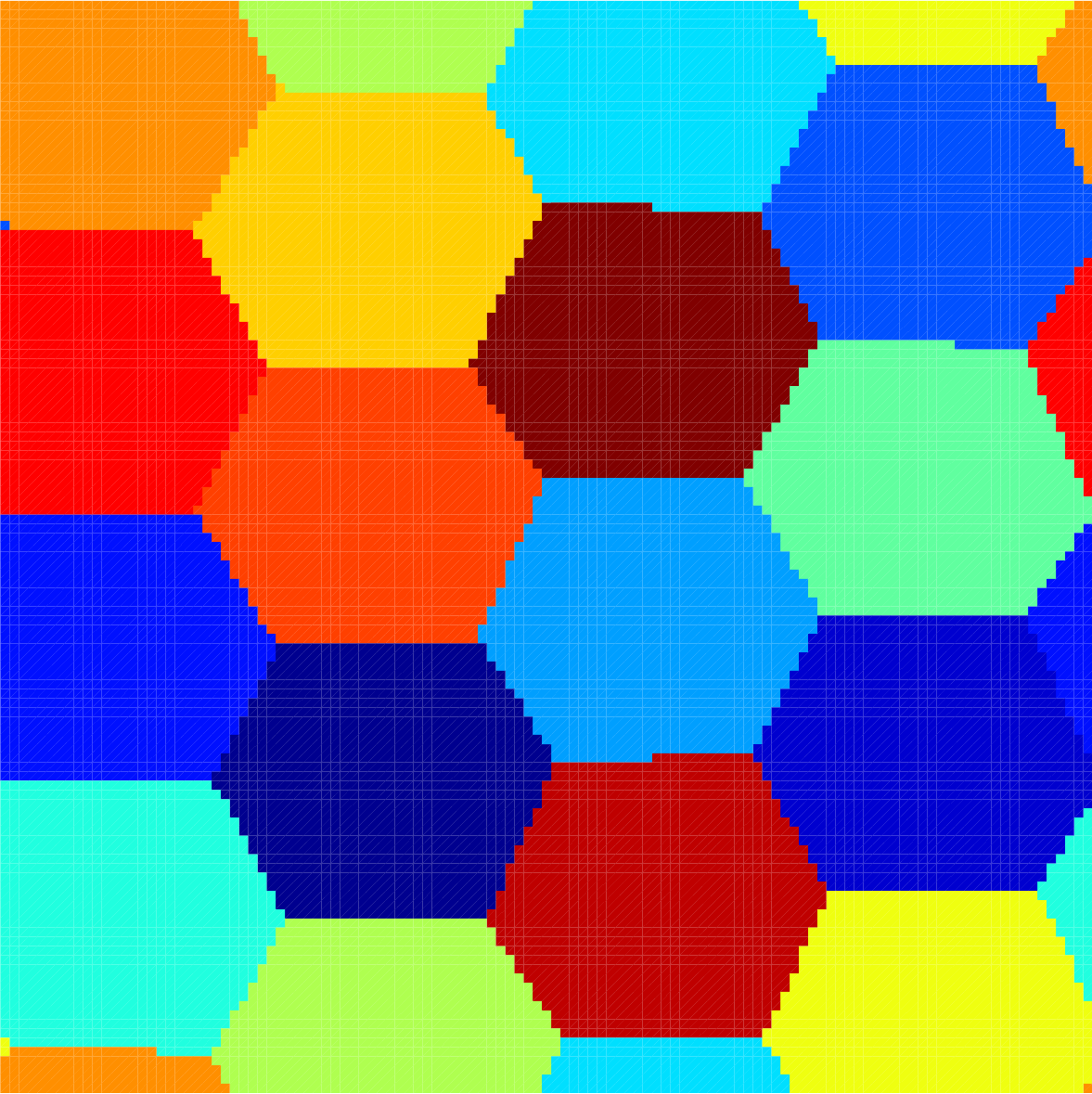}
\includegraphics[width=5cm]{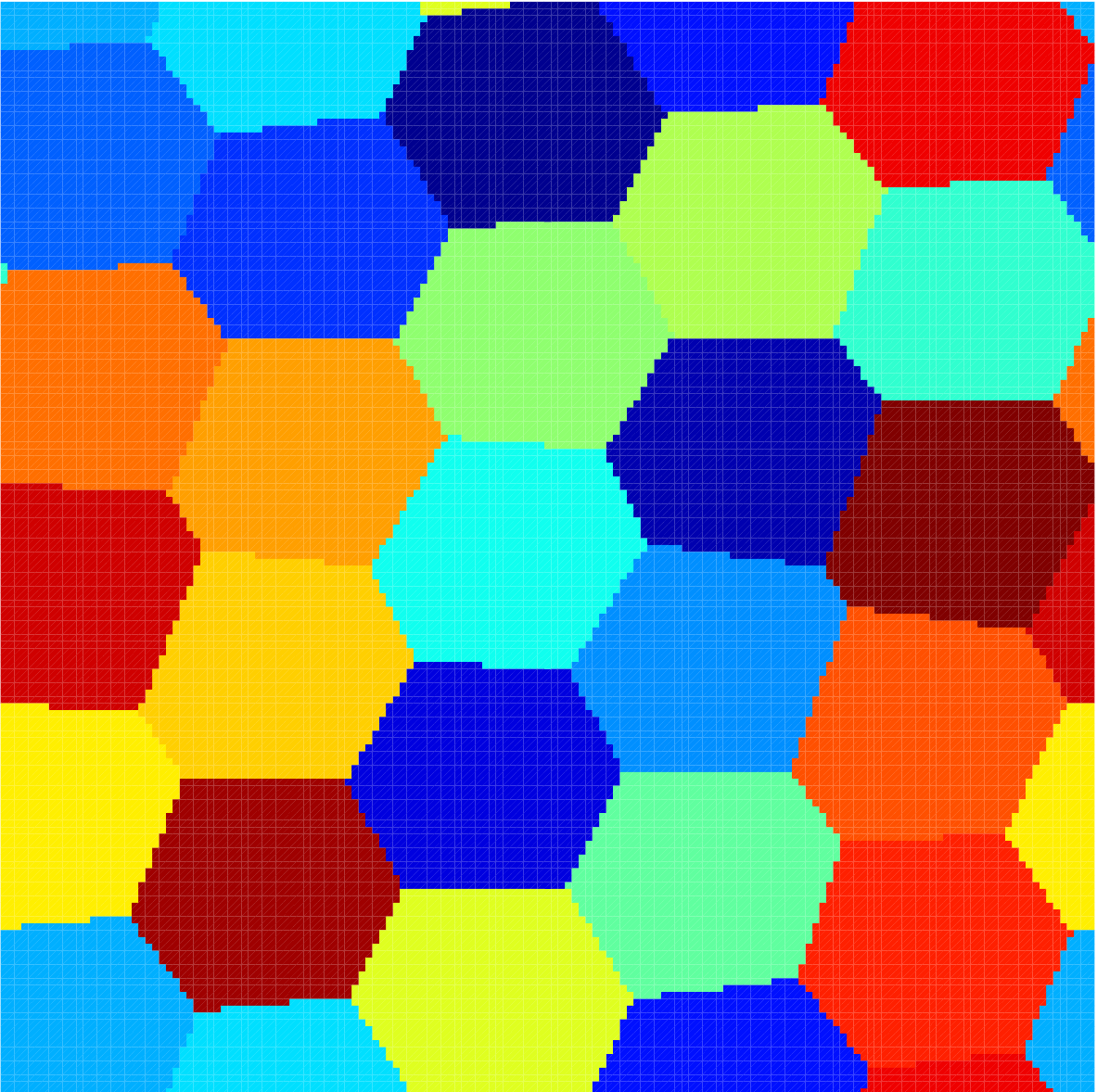}
\caption{Locally optimal $k=9$, 16, and 25 partitions of a flat torus. See \S \ref{sec:tori}.}
\label{fig:tori}
\end{center}
\end{figure}

\subsection{Sphere} \label{sec:sphere}
An approximately uniform  triangular   tessellation  of the sphere with $n=4000$ points is generated by minimizing the Reisz s-energy\footnote{\url{http://www.mathworks.com/matlabcentral/fileexchange/37004-uniform-sampling-of-a-sphere}}.  Using the geodesic distance, we use a standard Gaussian kernel $e^{-d^{2}(x,y)/\sigma^{2}}$ with $\sigma^2 = 1/10$ to construct the similarity matrix. The $r=0$ graph Laplacian is used for partitioning.  We choose $k=3$ and initialize the partition with a Voronoi tessellation for 5 instances of  randomly chosen generators. We choose $\alpha = 3 \lambda_{2}$ and the algorithm converges in roughly 13 iterations. In Fig. \ref{fig:yPart}, we plot the lowest energy partition. The conjectured Y-partition is attained; see \cite{Helffer2010}.

\begin{figure}[t!]
\begin{center}
\includegraphics[width=13cm]{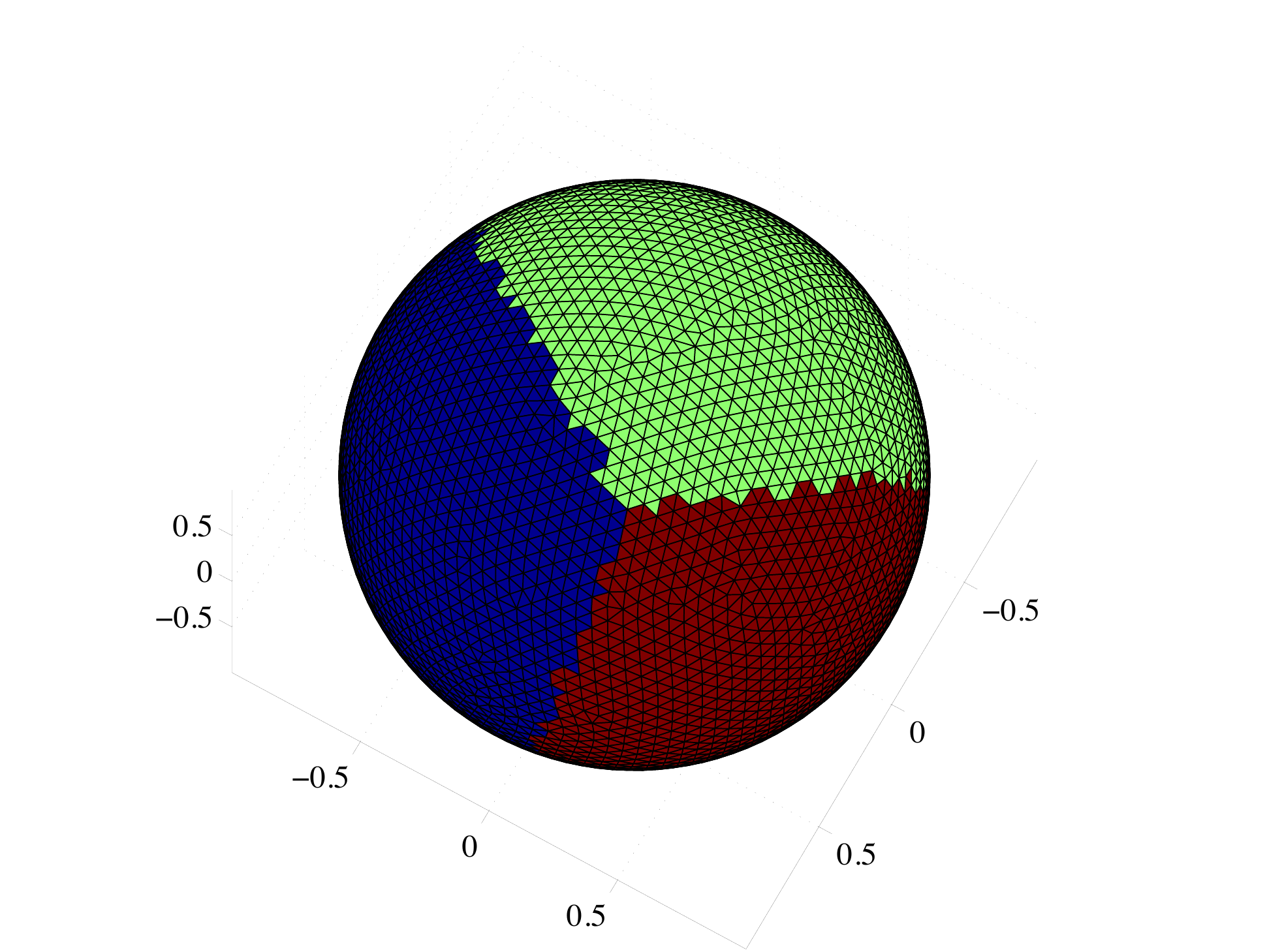}
\caption{A Y-partition of the sphere obtained via the rearrangement algorithm. See \S \ref{sec:sphere}.}
\label{fig:yPart}
\end{center}
\end{figure}

\section{Discussion and further directions} \label{sec:disc}

In this paper we introduced a new non-convex partitioning objective based on minimizing the sum of the Dirichlet eigenvalues of the partition components and constructed an algorithm for minimizing a relaxed version of this objective.  We believe this model is promising for many reasons; in particular, it does not rely on minimizing a functional of perimeter and thus is able to take the interior geometry of the partition components into account.  
Moreover, our algorithm naturally provides confidences for label assignments and 
consequently produces a representative for each cluster.
  Our model was motivated by an analogous continuous problem (\ref{eq:ContPart});  we have extended many well-established facts from the continuum to an arbitrary graph and, in some cases, proven considerably  more.  Many interesting research directions remain---perhaps most ambitiously, we wonder if it is possible to prove convergence (in the sense of Gromov-Hausdorff) of the discretized manifold graph partitions to their continuum limit.  We would also like to find a tractable relaxation of the similar Neumann boundary conditions problem, and analyze its properties. 

Another more immediate direction of research is in the numerical analysis of the rearrangement algorithm \ref{alg:Rearrangement}. The majority of the computational costs for the rearrangement algorithm are in the eigenvalue computations. Our implementation  could possibly be improved via a multi grid approach  \cite{Bucur1998}, the  Nystr\"om method \cite{Fowlkes2001}, the Rayleigh--Chebyshev method \cite{anderson2010}, or parallelization \cite{Oudet2011}.  In \S \ref{alpha_ad} we discuss optimal choices of the parameter $\alpha$, and believe optimal choices should be on the scale of the second eigenvalue of the graph Laplacian.  Another user-directed choice lies in which graph Laplacian to choose; we have found that optimal choices here are highly application dependent, and also dependent on the sparsity level of the given graph.  In a future work, we intend to analyze the numerics further, and will do a more systematic comparison against other algorithms, both unsupervised and semi-supervised.

\subsection*{Acknowledgements} 
We would  like to thank Andrea Bertozzi, James von Brecht, Dorin Bucur, Thomas Laurent, Stan Osher, Marc Sigelle, and David Uminsky for helpful conversations. Braxton Osting is supported in part by a National Science Foundation (NSF) Postdoctoral Fellowship DMS-1103959. Chris White was supported by AFOSR MURI grant FA9550-10-1-0569, ONR grant N000141210838 and ONR grant N000141210040; he would like to thank Rachel Ward for her support and Andrea Bertozzi for her hospitality in inviting him to UCLA.  
\'Edouard Oudet gratefully acknowledges the support of the ANR, through the projects GEOMETRYA and OPTIFORM.

%\clearpage
{\small 
\bibliographystyle{amsalpha}
\bibliography{graphPartition.bib}}

\end{document}